\documentclass[12pt,a4paper]{article}
\ifx\pdfpageheight\undefined\PassOptionsToPackage{dvips}{graphicx}\else%
\PassOptionsToPackage{pdftex}{graphicx}
\PassOptionsToPackage{pdftex}{color}
\fi

\usepackage{url,color}
\usepackage{verbatim}
\usepackage{latexsym}
\usepackage{amssymb,amstext,amsmath}
\usepackage{stmaryrd}
\usepackage{float} \floatstyle{boxed} \restylefloat{figure}
\usepackage{verbatim}
\usepackage{booktabs}
\usepackage{proof}
\usepackage{latexsym}
\usepackage{amssymb}
\usepackage{enumitem}
\usepackage{tikz-cd}

\ifx\pdfpageheight\undefined\relax\else%
\usepackage[hidelinks]{hyperref}
\fi

\usepackage{amsthm}
\theoremstyle{definition}
\newtheorem{theorem}{Theorem}[section]
\newtheorem{definition}[theorem]{Definition}
\newtheorem{lemma}[theorem]{Lemma}
\newtheorem{corollary}[theorem]{Corollary}
\newtheorem{remark}[theorem]{Remark}
\newtheorem{example}[theorem]{Example}
\newtheorem{discussion}[theorem]{Discussion}

\setlist{itemsep=.25em} % reduce vertical spacing in lists

\newcommand{\nat}{\mathbb{N}}
\newcommand{\XX}{\mbox{\sf X}}
\newcommand{\ndE}{\vdash_{\es}}
\newcommand{\ndX}{\vdash_{X}}

\newcommand{\ndY}{\vdash_{Y}}
\newcommand{\ndXx}{\vdash_{X,x}}
\newcommand{\ndXxi}{\vdash_{X,\vec x_i}}
\newcommand{\set}[1]{\{#1\}}
\newcommand{\es}{\emptyset}

\newcommand{\covd}{\mathrel{\vartriangleleft}}
\newcommand{\covs}{\mathrel{\vartriangleright}}
\newcommand{\covT}{\mathrel{\vartriangleleft_T}}

\newcommand*{\C}{{\mathbb{C}}}
\newcommand*{\D}{{\mathbb{D}}}

\newcommand*{\subvs}{_{\mathrm{vs}}}
\newcommand*{\subrn}{_{\mathrm{rn}}}
\newcommand*{\subts}{_{\mathrm{ts}}}
\newcommand*{\Cvs}{\mathbb{C}\subvs}
\newcommand*{\Crn}{\mathbb{C}\subrn}
\newcommand*{\Cts}{\mathbb{C}\subts}

\newcommand*{\Fvs}{\Vdash\subvs}
\newcommand*{\Frn}{\Vdash\subrn}

\newcommand*{\stdcat}[1]{\textup{\bfseries #1}}
\newcommand*{\Set}{\stdcat{Set}}
\newcommand*{\Mod}{\stdcat{Mod}}
\newcommand*{\Fin}{\stdcat{Fin}}
\newcommand*{\Finmon}{\Fin_{\mathrm{mon}}}
\newcommand*{\blank}{\textup{--}}
\newcommand*{\proves}{\mathrel{\vdash}}
\newcommand*{\forces}{\mathrel{\Vdash}}

\newcommand*{\sor}{\mathrel{\cup}}

\newcommand*{\op}{^{\mathrm{op}}}
\newcommand*{\denote}[1]{\llbracket #1 \rrbracket}

\DeclareMathOperator{\id}{id}
\DeclareMathOperator{\Hom}{Hom}

\DeclareMathOperator{\Sh}{Sh}
\DeclareMathOperator{\Psh}{PSh}

\DeclareMathOperator{\cod}{cod}
\DeclareMathOperator{\Tm}{Tm}
\DeclareMathOperator{\Fact}{Fact}

\newcommand*{\yo}{\mathbf{y}} % Yoneda
\newcommand*{\as}{\mathbf{a}} % associated sheaf
\newcommand*{\incl}{\mathbf{i}} % inclusion functor
\renewcommand*{\phi}{\varphi}

\begin{document}

\title{Syntactic Forcing Models for Coherent Logic}

\author{
Marc Bezem\thanks{Department of Informatics, University of Bergen,
{\tt bezem@ii.uib.no}}
\and
Ulrik Buchholtz\thanks{Department of Philosophy, Carnegie Mellon
  University \emph{and} Department of Mathematics, Technische Universit{\"a}t Darmstadt,
{\tt buchholtz@mathematik.tu-darmstadt.de}}
\and
Thierry Coquand\thanks{Department of Computer Science and Engineering, Chalmers/University of Gothenburg, {\tt coquand@chalmers.se}}
}
\date{\today}
\maketitle

\begin{abstract}
  We present three syntactic forcing models for coherent logic.  These
  are based on sites whose underlying category only depends on the
  signature of the coherent theory, and they do not presuppose that
  the logic has equality.
  As an application we give a coherent theory $T$ and a
  sentence $\psi$ which is $T$-redundant (for any
  geometric implication $\phi$, possibly with equality, if
  $T + \psi \proves \phi$, then $T \proves \phi$), yet false in the
  generic model of $T$. This answers in the
  negative a question by Wraith.
\end{abstract}

\section{Introduction}

Coherent logic concerns implications between positive formulas, those
built up from atoms using only the connectives $\top$, $\bot$,
$\lor$, $\land$ and $\exists$. A first-order theory $T$ is
\emph{coherent} if it is axiomatized by sentences of the form
$\forall\vec x.\, \phi\to\psi$ where $\phi,\psi$ are positive
formulas. Such sentences are also called \emph{coherent
  implications} or \emph{coherent sentences}.%
\footnote{NB\;Sometimes positive formulas are also called coherent formulas,
  but a coherent sentence in our sense need not be a coherent formula in this sense.}
Any coherent implication is equivalent to a finite
conjunction of sentences of the form
\begin{equation}\label{eq:coherent-implication}
  \forall \vec x.\,(\phi_0 \to  \exists\vec x_1.\phi_1 \lor\cdots\lor
  \exists\vec x_k.\phi_k),
\end{equation}
where the $\phi_i$ are conjunctions of atoms, and we can and will thus
always assume that any coherent theory $T$ is presented by axioms of this
form.

Coherent theories include all universal Horn theories. The axioms of equality
are all coherent implications, and over logic with equality all
algebraic theories, as well as the theories of fields and local rings,
are coherent. Coherent implications form a \emph{Glivenko class},
i.e., if a coherent implication is derivable from a coherent theory
using classical logic, then it is already so derivable
intuitionistically. Furthermore, they are preserved by the inverse
image parts of geometric morphisms between toposes, and for every
coherent theory $T$ there is a \emph{generic model} $M_T$ in a sheaf
topos $\Set[T]$, called the \emph{classifying topos} of $T$,
classifying models of $T$ in any sheaf topos (cf.\
Section~\ref{sec:classifying-toposes} for precise definitions). The
geometric perspective offered by sheaf toposes also motivates the
introduction of \emph{geometric theories}, those axiomatized by
\emph{geometric implications}, i.e., implications between infinitary positive formulas, or equivalently,
those with axioms of the form~\eqref{eq:coherent-implication} where
the disjunction is allowed to be infinite (as usual when dealing with
infinitary fragments, we require that all formulas only have finitely
many free variables).
We assume the existence of a Grothendieck universe or some other means of
talking about small sets. We then only consider infinitary conjunctions and
disjunctions indexed by small sets.

The generic model of a coherent theory $T$ can be thought of as a
forcing model, and in this paper we present three other forcing models
for coherent logic without equality, providing proofs of soundness and
completeness. The models can be understood as living in certain sheaf
toposes, and in a companion paper we shall characterize these as
classifying toposes of certain geometric theories related to $T$.

A main goal of the current paper is to answer the following question
of Wraith:
\begin{quote}
\itshape The problem of characterising all the non-geometric properties of a generic
model appears to be difficult. If the generic model of a geometric theory $T$ satisfies
a sentence $\alpha$ then any geometric consequence of  $T+(\alpha)$ has to be a
consequence of $T$. We might call $\alpha$~~$T$-redundant. Does the generic
$T$-model satisfy all $T$-redundant sentences?
\normalfont
\hfill\cite[p.~336]{Wraith1980}
\end{quote}
We shall answer this in the negative, even for a coherent theory $T$.
For the current volume in honour of Brouwer it is appropriate to clarify
that the question makes constructive sense, and that our answer is
constructive too.
The proper understanding of Wraith's question requires a fair amount
of categorical logic,  which we develop in Section~\ref{sec:PrelimSheafTopos}.
For the construction of the generic model we rely on
Coste and Coste~\cite{CosteCoste1975}. These preparations postpone
the proof of the negative answer to Wraith's question until the very last
section of this paper. Therefore we find it useful to sketch this proof
already here.

We shall give a consistent coherent theory $T$ and a sentence $\psi$
such that both $\psi$ and $\neg\psi$ are $T$-redundant.
Since the generic model of $T$ cannot satisfy both $\psi$ and $\neg\psi$,
this provides a negative answer to Wraith's question.
However, this answer is not as informative as one would hope:
one would like to know which one of $\psi$ and $\neg\psi$
is true in the generic model of $T$. For this we have to take
closer look at the argument why both $\psi$ and $\neg\psi$ are $T$-redundant.

Soundness for our forcing models means that every intuitionistic consequence (coherent or
not, possibly with $=$) of $T$ is forced.
In fact, we prove soundness for all infinitary formulas,
including all geometric implications.
Geometric completeness means
that any (generalized) \emph{geometric implication without $=$} that is forced
is an intuitionistic consequence of the coherent theory.

Now let $\psi$ be any sentence without $=$
that is forced in any one of our models.
Let $\phi$ be
a geometric implication without $=$ such that $\psi\to\phi$
is intuitionistically provable in $T$.
Then by soundness $\psi\to\phi$ is forced,
and hence $\phi$ is forced.
Since $\phi$ is geometric without $=$ it follows by geometric completeness
that $\phi$ is provable in $T$ \emph{without using} $\psi$.
A sentence like $\psi$ is thus $T$-redundant in logic without equality.
It can then be shown by a  cut-elimination argument
that $\psi$ is also $T$-redundant in logic with equality.

The actual example we give of a coherent theory $T$ and
$T$-redundant sentences $\psi$ and $\neg\psi$ does not involve $=$
and is moreover \emph{relational}, that is, in a signature without functions.
The $T$-redundancy follows from $\Fvs\psi$ and $\Frn\neg\psi$, using
two of our forcing models.
The final step in the constructive answer to Wraith's question is the observation
that, due to the relational signature, $\Fvs$ exactly characterizes truth in the
generic model of $T$. In other words, the example sentence $\neg\psi$ is
$T$-redundant, yet false in the generic model.

The rest of the paper is organized as follows. In
Section~\ref{sec:PrelimSheafTopos} we recall in a mostly
self-contained way some preliminaries on site models in sheaf
toposes. Then in Section~\ref{sec:basics} we introduce our
site models, which we then compare to the generic model
in Section~\ref{sec:comparisons}. We carry out
a uniform completeness proof in
Section~\ref{sec:coherent_completeness}, and we give examples in
Section~\ref{sec:examples}.

We would like to emphasize that although we include a fair amount of
categorical logic in Sections~\ref{sec:PrelimSheafTopos}
and~\ref{sec:comparisons}, our forcing models and the accompanying
soundness and completeness results can be understood without this
machinery. Moreover, all results are constructively valid.

\section{Preliminaries on sheaf toposes}\label{sec:PrelimSheafTopos}

We shall in this section
briefly recall the main definitions and some auxiliary results
concerning models in sheaf toposes.
We refer to \cite{Elephant,MacLaneMoerdijk1994} for detailed
expositions.

We begin with the notion of a Grothendieck topology on a small
category $\C$ and ways to
present these. We write $\yo : \C \to \Psh(\C)$ for the Yoneda
embedding of $\C$ into the \emph{presheaf topos}
$\Psh(\C)$ consisting of functors $\C\op\to\Set$:
$\yo(C)=\Hom_\C(\blank,C)$. If $P$ is a presheaf, $f : D \to C$ in
$\C$ and $x\in P(C)$, we write $xf$ for the action $P(f)(x)$ of $f$ on
$x$ so that if also $g : E \to D$, then $x(fg)=(xf)g$.

A \emph{sieve} on an object $C$ is a
subobject $S\subseteq\yo(C)$ in $\Psh(\C)$, viz., a collection of
morphisms in $\C$ with codomain $C$ closed under precomposition with
arbitrary morphisms. A set of arrows $U$ with codomain $C$ is called a
\emph{$C$-sink}; in particular, we think of sieves on $C$ as certain
$C$-sinks. For $g : D \to C$ a morphism with codomain $C$ and $S$ a
$C$-sink, we define the $D$-sink $g^*(S)$ to be the set $\set{h \mid
  \cod(h)=D \land gh\in S}$. This is a sieve whenever $S$ is.

\begin{definition}\label{def:Gtop}
  A \emph{Grothendieck topology} on $\C$ is a function $J$ that
  assigns to every object $C$ a collection $J(C)$ of sieves on $C$,
  such that
  \begin{enumerate}[label=(\roman*)]
  \item the maximal sieve $\{\, f \mid \cod(f)=C\,\}$ denoted by $t_C$ is in
    $J(C)$;
  \item (stability) if $S\in J(C)$ and $g:D\to C$, then $g^*(S)$
    is in $J(D)$;
  \item (transitivity) if $S\in J(C)$ and $R$ is any sieve on $C$ such
    that $f^*(R)\in J(D)$ for every $f:D\to C$ in $S$, then $R$ is in
    $J(C)$.
  \end{enumerate}
\end{definition}

A category $\C$ with a Grothendieck topology is called a
\emph{site}.
If $U$ and $V$ are $C$-sinks,
we say that $U$ \emph{refines} $V$ if every $f\in U$ factors through
some $g\in V$.
A $C$-sink $U$ is called a \emph{$J$-cover} if there
exists $S \in J(C)$ refining $U$.

It is often useful to present a topology using a basis or
a (saturated) coverage:
\begin{definition}\label{def:Gbasis}
  A \emph{basis} for a Grothendieck topology on a category $\C$ with
  finite limits consists of a relation $\covd$ between objects $C$ of
  $\C$ and $C$-sinks $U$ (if $C \covd U$ we say that $U$ is a
  \emph{basic cover} of $C$) such that the following conditions hold:
  \begin{enumerate}[label=(\roman*)]
  \item\label{def:Gbasis-i} if $f : D \to C$ is an isomorphism, then $C \covd \{f\}$;
  \item\label{def:Gbasis-ii} if $C \covd U$ and $g : D \to C$, then $D \covd \set{g^*f \mid f\in U}$;
%  \item if $C \covd U$ and $f \covd U_f$ for each $f\in U$, then
  \item\label{def:Gbasis-iii} if $C \covd U$ and $D \covd V_f$ for each $f: D\to C$ in $U$, then
    $C \covd \bigcup_{f\in U} {f V_f}$.
  \end{enumerate}
  Here and below $g^* f$ is the pullback of $f$ along $g$.
  % $f \covd U_f$ means $D \covd f^*(U_f)$ if $f : D \to C$.}
  In case $\C$ does not admit finite limits, we have to replace
  \ref{def:Gbasis-ii} by a weaker condition, namely that of a
  coverage. A \emph{coverage} on a small category $\C$ consists of a
  relation $\covd$ between objects $C$ of $\C$ and $C$-sinks $U$ such that
  \begin{enumerate}[label=(ii')]
  \item\label{def:Gbasis-ii'} if $C \covd U$, and $g : D \to C$, then
    there exists $V \covs D$ such that $gV$ refines $U$.
  \end{enumerate}
  If a coverage also satisfies the
  conditions \ref{def:Gbasis-i} and \ref{def:Gbasis-iii}, then it will
  be called a \emph{saturated coverage}.\footnote{Beware that this
    terminology is not quite standard: a ``saturated
    coverage'' is sometimes required to satisfy as well (iv) if
    $U\covs C$ refines a $C$-sink $V$, then already $V\covs
    C$. Saturated coverages in this sense are in bijection with
    Grothendieck topologies.}
  Note that every basis is a saturated coverage by choosing
  $V=\set{g^*f \mid f\in U}$ in \ref{def:Gbasis-ii'}.
\end{definition}

The reader may check that every saturated coverage \emph{generates} a
topology by letting $J(C)$ consist of those $S$ for which there exists
$U\covs C$ such that $U\subseteq S$ (i.e., $U$ refines $S$, by the
closure property of a sieve $S$).

\begin{lemma}\label{lem:Gcommonrefinement}
Let $\covd$ be a saturated coverage on $\C$. If $C\covd U$ and $C\covd V$,
then $U$ and $V$ have a common refinement.
\end{lemma}
\begin{proof}
Let $C\covd U,V$. By \ref{def:Gbasis-ii'} there exists for every $g:D\to C$ in $V$
a $W_g \covs D$ such that $g W_g$ refines $U$. Then by \ref{def:Gbasis-iii} we have
$\bigcup_{g\in V} g W_g \covs C$ refining both $U$ and $V$ per construction.
\end{proof}

\begin{definition}\label{def:Gcompat}
  Let $\C$ be a category and $U$ a $C$-sink for some
  $C\in\C$. Write $U=\{f_i:C_i \to C\}_{i\in I}$. Let $P:\C\op\to\Set$ be a
  presheaf. A \emph{compatible family} of elements $s_i\in P(C_i)$ for
  $U$ is one such that whenever $g : D \to C_i$ and $h : D \to C_j$
  satisfy $f_ig=f_jh$, then $s_ig=s_jh$.
\end{definition}

It is clear that if $\C$ has pullbacks, then it suffices to check
compatibility at each pullback corner $C_i \times_C C_j$ (it is of
course necessary, too).  A matching
family for a sieve $S$ is just a compatible family, but in this case
it suffices to check that $x_f g=x_{fg}$ for all $f\in S$ and all
$g$. This just amounts to a natural transformation $S\to P$.

\begin{definition}\label{def:Gsheaf}
  A presheaf $P : \C\op \to \Set$ is a \emph{sheaf} for the coverage $\covd$,
  if for every object $C\in\C$, every $U \covs C$, and every
  compatible family of elements $(s_f)_{f\in U}$ for $U$, there exists a unique
  element $s\in P(C)$ such that $s_f=sf$ for all $f\in U$.
\end{definition}
By an adaptation of the argument of \cite[Prop.\
III.4.1]{MacLaneMoerdijk1994} we have for a coverage satisfying
\ref{def:Gbasis-iii} that $P$ is a sheaf for $\covd$ if{f}
$P$ is a sheaf for the generated topology, in the sense that whenever
$S \hookrightarrow \yo C$ is a covering sieve and $S \to P$ is a
compatible family, there is a unique morphism $\yo C\to P$ (i.e., an
element of $P(C)$) making the triangle commute:
\[\begin{tikzcd}
    S \arrow{r}\arrow[hookrightarrow]{d} & P \\ \yo C\ar[dashed]{ur} &
  \end{tikzcd}\]

\begin{definition}
  For a Grothendieck topology $J$ or a saturated coverage $\covd$ on
  $\C$, let $\Sh(\C,J)$, resp.\ $\Sh(\C,\covd)$, be the full subcategory
  of $\Psh(\C)$ consisting of sheaves wrt.\ $J$, resp.\ $\covd$. If
  $J$ or $\covd$ can be inferred from the context, we just write
  $\Sh(\C)$. A category equivalent to one of this form is called a
  \emph{Grothendieck topos} or a \emph{sheaf topos}.
\end{definition}

\begin{example}\label{exa:presheaftopos}
  For any small category $\C$, the topology $J$ consisting just of the
  maximal sieves, $J(C)=\{t_C\}$, is called the \emph{trivial
    topology}. It is generated by the saturated coverage $\covs$ where
  $U\covs C$ if and only if $U=\{f\}$ with $f:C'\to C$ an
  isomorphism. In this way every presheaf is a sheaf, so
  $\Psh(\C)=\Sh(\C)$, the \emph{presheaf topos}, is indeed a sheaf
  topos.
\end{example}

The inclusion functor $\incl:\Sh(\C,J)\hookrightarrow\Psh(\C)$ has a left
adjoint, $\as:\Psh(\C)\to\Sh(\C,J)$, called \emph{the associated
  sheaf} functor, cf.\ \cite[Sect.~III.5]{MacLaneMoerdijk1994}.%
\footnote{\label{ft:sheafification} This functor is also called
  \emph{sheafification}. We note that its definition is
  constructive for topologies generated by coverages with
  finite covers. Otherwise, known definitions require either choice
  for the sets indexing the covers or impredicative quantification.}
The unit of the adjunction $\eta$ is a natural transformation with components
$\eta_P : P \to \as P$ for each presheaf $P$ (we write $\eta$ for
$\eta_P$ when no confusion can arise). Elements of $\as P(C)$ are
given by equivalences classes of \emph{locally compatible families}
for covers $U\covs C$: families of elements $s_i\in P(C_i)$ for each
$f_i:C_i \to C$ in $U$ such that for every $i$, $j$, and $g:D\to C_i$ and
$h:D\to C_j$ such that $f_ig=f_jh$, the elements $s_ig=s_jh$ are
\emph{locally equal}. The latter means that there is
is a further cover $V\covs D$ such that $s_igk=s_jhk$ for every $k\in V$.

Two locally compatible families are equivalent if they agree \emph{locally}
on a cover that refines the corresponding covers. More precisely,
if $s_i\in P(C_i)$ and $t_j\in P(C_j)$ are locally compatible families
for $U\covs C$ and $V\covs C$, respectively, then they are equivalent
if there exists $W\covs C$ refining both $U$ and $V$, such that for
all $h: D\to C$ in $W$ and any way $h$ factors through $U$ and through $V$,
say $h = f_i h_i = g_j h_j$ with $f_i \in U$ and $g_j \in V$,
we have that $s_i h_i$ and $t_j h_j$ are locally equal in $D$. Note that the
particular $i$ and $j$ depend on the factorization and that there
is at least one of each for each $h$.

\subsection{Geometric morphisms}

The natural notion of a morphism between sheaf toposes
$\mathcal E, \mathcal F$ is that of a \emph{geometric morphism}
$F:\mathcal E\to\mathcal F$. This is a pair of adjoint functors,
\[
  \begin{tikzcd}
    \mathcal E \arrow[shift right=1]{r}[swap]{F_*} &
    \mathcal F \arrow[shift right=1]{l}[swap]{F^*}
  \end{tikzcd}, \quad F^* \dashv F_*,
\]
such that $F^*$ preserves finite limits. Here, $F^*$ is called the
\emph{inverse image} functor and $F_*$ is called the \emph{direct
  image} functor. The geometric morphisms from $\mathcal E$ to
$\mathcal F$ form a category, $\Hom(\mathcal E,\mathcal F)$ where the
morphisms are natural transformations between the inverse image parts
(or equivalently between the direct image parts in the opposite
direction). Cf., e.g., Chapter~VII of \cite{MacLaneMoerdijk1994} for
more information about geometric morphisms.

The first example of a geometric morphism is the inclusion of the
sheaf topos into the presheaf topos on a site $\C$, $(\as \dashv \incl):
\Sh(\C) \to \Psh(\C)$. The direct image functor is the inclusion
functor and the inverse image functor is the associated sheaf functor,
which is indeed left exact (in particular, it preserves
monomorphisms).

We now wish to consider geometric morphisms between sheaf toposes
induced by functors between the underlying sites. Recall that any
functor $F : \C \to \D$ between small categories induces adjunctions
\[
  \begin{tikzcd}
    \Psh(\C) \arrow[shift left=0.6em]{r}{F_!}\arrow[shift right=0.6em]{r}[swap]{F_*} &
    \Psh(\D) \arrow{l}[description]{F^*}
  \end{tikzcd}
\]
with $F_! \dashv F^* \dashv F_*$. Here, $F^*$ is just composition with
$F$, while $F_!$ and $F_*$ are the left and right Kan extensions along
$F$. Since $F^*$ is a right adjoint, it preserves all limits, so the
pair $(F^* \dashv F_*)$ gives a geometric morphism
$F : \Psh(\C) \to \Psh(\D)$.

We first consider the conditions under which this induces a geometric
morphism $F : \Sh(\C) \to \Sh(\D)$ when $\C$ and $\D$ are sites.
\begin{definition}\label{def:site-comorphism}
  A functor $F : \C \to \D$ between sites is a \emph{comorphism of
    sites} if it has the following:
  \begin{description}
  \item[Covering lifting property] For any cover $V \covs F(C)$ in
    $\D$ with $C \in \C$, there is a cover $U \covs C$ in $\C$ such
    that $F(U)$ refines $V$.
  \end{description}
\end{definition}
The following result is essentially \cite[C2.3.18]{Elephant}, where
the above property is called \emph{cover-reflecting} (see also
\cite[Theorem~VII.10.5]{MacLaneMoerdijk1994}).

\begin{theorem}\label{thm:comorphism}
  Let $F : \C \to \D$ be a functor between sites. Then the following
  are equivalent:
  \begin{enumerate}[label={\normalfont(}\roman*\/{\normalfont)}]
  \item\label{thm:comorphism-i} $F$ is a comorphism.
  \item\label{thm:comorphism-ii} There is a geometric morphism
    $\Sh(\C)\to\Sh(\D)$ whose inverse image functor is the composite
    \[ \Sh(\D) \xrightarrow{i}{} \Psh(\D)
      \xrightarrow{F^*}{} \Psh(\C) \xrightarrow{\as}{} \Sh(\C).\]
  \item\label{thm:comorphism-iii} $F_* : \Psh(\C)\to\Psh(\D)$ maps
    sheaves to sheaves.
  \end{enumerate}
\end{theorem}
For use in Section~\ref{sec:comparison_morphisms} we recall the
following basic result about comorphisms of sites, which follows from
\cite[C2.3.18]{Elephant}:
\begin{lemma}\label{lem:comorphism_as}
  If $Q \in \Psh(\D)$ is a presheaf, and $F : \C \to \D$ a comorphism
  of sites, then the natural map
  $\as F^*(\eta_Q) : \as F^*(Q) \to \as F^*(\as Q)$ is an isomorphism
  in $\Sh(\C)$.
\end{lemma}

Sometimes a functor $F : \C \to \D$ induces a geometric morphism
$\Sh(\D) \to \Sh(\C)$ in a \emph{contravariant} way via the
adjoint pair $(F_! \dashv F^*)$. In this case, the left exactness of
$F_!$ is not automatic, which is one reason the corresponding
definition has two components:

\begin{definition}\label{def:site-morphism}
  A functor $F : \C \to \D$ between sites is a \emph{morphism of
    sites} if it satisfies the following two properties:
  \begin{description}
  \item[Cover preserving] If $U \covs C$ in $\C$, then there is a
    cover $V \covs FC$ refining $F(U)$.
  \item[Covering-flat] If $A : I \to \C$ is a finite diagram in $\C$,
    then every cone over $FA$ factors locally through the $F$-image of
    a cone over $A$.
  \end{description}
\end{definition}

The notion of being covering-flat can be slightly strengthened by
removing the word ``locally'' from the definition. A functor
satisfying this condition is called \emph{representably flat}. Of
course, if $\C$ and $\D$ have finite limits and $F$ preserves them,
then $F$ is representably flat, and a fortiori covering-flat.
See also the discussion in \cite{Shulman2012}, where we find:

\begin{theorem}\label{thm:site-morphism}
  If $F : \C \to \D$ is a morphism of sites, then there is geometric
  morphism $\Sh(\D) \to \Sh(\C)$ whose direct image functor is the
  restriction of $F^* : \Psh(\D) \to \Psh(\C)$ to sheaves.
\end{theorem}

\subsection{Structures for a signature}\label{sec:structures}

We fix a first-order signature $\Sigma$ throughout the sequel. We give
all definitions for the case of a single-sorted signature, leaving the
details concerning the many-sorted case to the reader.

\begin{definition}
  A \emph{structure} for $\Sigma$ in a sheaf topos $\Sh(\C,J)$
  consists of a sheaf
  $M$ (the carrier) together with morphisms $\denote f_M : M^n\to M$ for every
  $n$-ary function symbol, and subobjects of $M^n$ denoted
  $\denote P_M \hookrightarrow M^n$ for every $n$-ary predicate symbol.
  The latter two are often left implicit when we speak about a structure
  or model $M$.

  A \emph{homomorphism} of $\Sigma$-structures is a morphisms $g : M
  \to M'$ that respects the interpretations of function and predicate
  symbols.
\end{definition}

Any formula $\phi(\vec x)$ with free variables among
the $\vec x=x_1,\dots,x_n$, gives rise to a subobject
$\denote{\phi(\vec x)}_M \hookrightarrow M^n$, defined by recursion
on the structure of $\phi$ using the structure ingredients in the
base cases. This semantics has an explicit description in terms of the
\emph{forcing} relation:

\begin{definition}\label{def:forcing_structure}
  Given an object $C\in\C$ and a tuple of elements $\alpha\in
  M(C)^n = M^n(C)$. If $\phi(\vec x)$ is a formula with $n$ free variables,
  then we say that $C$ \emph{forces} $\phi(\alpha)$,
  denoted $C \forces_M \phi(\alpha)$, if{f} $\alpha$ factors through
  $\denote{\phi(\vec x)} \hookrightarrow M^n$, where we consider
  $\alpha$ as a map $\as\yo C \to M^n$.
\end{definition}

The forcing relation enjoys the following two properties:
\begin{description}
\item[Monotonicity] If $C \forces_M \phi(\alpha)$ and $f:D\to C$,
  then $D \forces_M \phi(\alpha f)$.
\item[Local Character] If $U\covs C$ is such that $C' \forces_M
  \phi(\alpha f)$ for every $f:C'\to C$ in $U$, then $C \forces_M
  \phi(\alpha)$.
\end{description}
We now have the following result, which is a slight variation of
\cite[Theorem~VI.7.1]{MacLaneMoerdijk1994}, using the fact that any
cover for a topology generated by a saturated coverage is refined by a
cover in the coverage (by definition).

\begin{theorem}\label{thm:forcing}
  Given a structure $M$ in a sheaf topos $\Sh(\C,\covd)$, where
  $\covd$ is a saturated coverage on $\C$, we have:
  \begin{itemize}
  \item $C\forces_M \top$ always;
  \item $C\forces_M \bot$ if{f} there is a cover $U\covs C$ with
    $U$ empty;
  \item $C\forces_M \bigwedge_{i\in I} \phi_i(\alpha)$ if{f}
    $C\forces_M\phi_i(\alpha)$ for all $i\in I$;
  \item $C\forces_M \bigvee_{i\in I} \phi_i(\alpha)$ if{f} there is
    a cover $U\covs C$ such that for every $f:C'\to C$ in $U$,
    there is some $i\in I$ such that $C'\forces_M \phi_i(\alpha f)$;
  \item $C\forces_M \phi(\alpha) \to \psi(\alpha)$ if{f} for
    every $f:C'\to C$ we have $C'\forces_M\psi(\alpha f)$ if
    $C'\forces_M\phi(\alpha f)$;
  \item $C\forces_M \forall x.\phi(x,\alpha)$ if{f} for every
    $f:C'\to C$ and every $\beta\in M(C')$,
    $C'\forces_M\phi(\beta,\alpha f)$;
  \item $C\forces_M \exists x.\phi(x,\alpha)$ if{f} there is a
    cover $U\covs C$ such that for every $f:C'\to C$ in $U$, there is
    some $\beta\in M(C')$ such that $C'\forces_M\phi(\beta,\alpha
    f)$.
  \end{itemize}
\end{theorem}

For a sentence $\phi$, we write $\forces_M \phi$ if
$C\forces_M \phi$ for all $C\in\C$. Of course, if $\C$ has a
terminal object $1$, then this is equivalent to
$1 \forces_M \phi$.
\begin{definition}
  Given a theory $T$ and a structure $M$ in $\Sh(\C,\covd)$, we say
  that $M$ is a \emph{model} of $T$ if $\forces_M \phi$ for every
  sentence $\phi$ in $T$.
\end{definition}

It is known that the forcing semantics is sound
\cite[D1.3.2]{Elephant} and complete \cite[D1.4.11]{Elephant} for
first-order intuitionistic logic (since sheaf toposes are Heyting
categories). In fact, one can construct a single sheaf topos with a
model $M$ of the first-order theory $T$ such that $\forces_M\phi$
if{f} $\proves_T\phi$ for all sentences $\phi$
\cite[Corollary~5.6]{Palmgren1997}. We give a similar completeness
proof for coherent logic below in
Theorem~\ref{thm:completeness}.

For future use we record the following refinements of the quantifier
cases of Theorem~\ref{thm:forcing}.
\begin{lemma}\label{lem:forcing-quant}
  If $M=\as P$ for some presheaf $P$, then we have
  for every $\alpha : \as\yo C \to M^n$ :
  \begin{itemize}
  \item $C\forces_M \forall x.\phi(x,\alpha)$ if{f} for every
    $f:C'\to C$ and every $\beta\in P(C')$,
    $C'\forces_M\phi(\eta_P(\beta),\alpha f)$;
  \item $C\forces_M \exists x.\phi(x,\alpha)$ if{f} there is a
    cover $U\covs C$ such that for every $f:C'\to C$ in $U$, there is
    some $\beta\in P(C')$ such that
    $C'\forces_M\phi(\eta_P(\beta),\alpha f)$.
  \end{itemize}
  Furthermore, if $P=\yo D$ is representable and $\C$ has finite
  products, then:
  \begin{itemize}
  \item $C\forces_M \forall x.\phi(x,\alpha)$ if{f}
    $C\times D\forces_M\phi(\pi_2,\alpha\pi_1)$,
    where $\pi_2\in\as\yo D(C\times D)$ is the element corresponding
    to the projection $C\times D\to D$.
    Also, $\alpha\pi_1 : \as\yo (C\times D)\to M^n$.
  \end{itemize}
\end{lemma}
\begin{proof}
  The first points follow from local character of forcing, together
  with the fact that every element of $\as P(C')$ is locally of the
  form $\eta_P(\gamma)$ for some $\gamma\in P(C'')$ (and transitivity
  of covers in the case of the existential quantifier).

  The last point is Remark~VI.7.2 in \cite{MacLaneMoerdijk1994}.
\end{proof}

If $f : \mathcal E \to \mathcal F$ is a geometric morphism of sheaf
toposes, and $M\in\mathcal F$ is a model of a coherent theory $T$,
then $f^*M$ is a model of $T$ in $\mathcal E$. This illustrates the
important relationship between coherent (and more generally,
geometric) theories and toposes and geometric morphisms. We shall
return to this relationship in the next subsection and in
Section~\ref{sec:comparisons}.

\subsection{Generalized geometric implications}
\label{sec:generalized-geometric}

The following variation on the class of geometric implications will be
used in our completeness proofs in
Section~\ref{sec:coherent_completeness}.

\begin{definition}\label{def:generalized-geometric}
  The class of \emph{generalized geometric implications} is generated by
  the following grammar:
  \[
    \phi ::= \alpha \mid (\phi_1 \land \phi_2)
    \mid \bigl(\bigvee\nolimits_{i\in I} \phi_i\bigr) \mid (\exists x.\phi)
    \mid (\forall x.\phi)
    \mid (\alpha \to \phi)
  \]
  where $I$ ranges over small index sets and
  $\alpha ::= p(\vec t\;) \mid \top \mid \bot \mid (s=t)$ is an atomic formula.
\end{definition}

Note that geometric implications are not strictly speaking generalized
geometric implications, but easy equivalents can be obtained via
the equivalence of $\alpha_1 \land \cdots \land \alpha_n \to \phi$ and
$\alpha_1 \to \cdots \to \alpha_n \to \phi$.

Also note that there are (even finitary) generalized geometric implications that are not
equivalent to any geometric implication, e.g., $p \lor (p \to \bot)$
for atomic $p$.

Among first-order formulas, the generalized geometric implications are
up to equivalence those that do not contain negative occurrences of
implications nor of universal quantifiers,
cf.~\cite[Theorem~3.2]{Negri2014} where the term
\emph{generalized geometric implications} refers to first-order
formulas.

\begin{theorem}\label{thm:generalized-preserved}
  Let $F:\mathcal E\to\mathcal F$ be a geometric morphism, and
  $M \in \mathcal F$ a $\Sigma$-structure. For any generalized
  geometric implication $\phi$ in variables $x_1,\dots,x_n$, we
  have an inclusion of subobjects of $(F^*M)^n
  \simeq F^*(M^n)$:
  \[
    F^* \denote\phi_M \le \denote\phi_{F^*M}.
  \]
  In particular, if $\phi$ is a sentence and $\denote\phi_M=\top$,
  then $\denote\phi_{F^*M}=\top$.
\end{theorem}
\begin{proof}
  By induction on $\phi$. The inverse image functor preserves atomic
  formulas, conjunction, disjunction, and existential
  quantification. Hence it remains to consider implication and
  universal quantification.

  Note that if $f : P \to Q$ is any morphism in $\mathcal F$ and $A$ a
  subobject of $P$, then $F^*(\forall_f A)\le \forall_{F^*f}(F^*A)$,
  cf.~\cite[IX.6.(17)]{MacLaneMoerdijk1994}. And if $B$ is another
  subobject of $P$, then implication can be expressed as $(A
  \Rightarrow B) = \forall_a(A\land B)$, where $a : A \to P$ is the
  inclusion of the subobject $A$. Hence $F^*(A \Rightarrow B) =
  F^*(\forall_a.A\land B) \le \forall_{F^*a}(F^*A \land F^*B)
  = (F^*A \Rightarrow F^*B)$.
  
  For an implication $\alpha \to \phi$, we now have
  \begin{align*}
    F^* \denote{\alpha\to\phi}_M
    &= F^*(\denote\alpha_M \Rightarrow \denote\phi_M)
    \le (F^*\denote\alpha_M \Rightarrow F^*\denote\phi_M) \\
    &= (\denote\alpha_{F^*M} \Rightarrow F^*\denote\phi_M)
    \le (\denote\alpha_{F^*M} \Rightarrow \denote\phi_{F^*M})
    = \denote{\alpha\to\phi}_{F^*M},
  \end{align*}
  using the fact that atomic formulas are preserved.
  The case for a universal quantifier is similar.  
\end{proof}

\section{Categories of forcing conditions}\label{sec:basics}

In this section we define three categories of forcing conditions
and do the groundwork for their sheaf toposes based on a coherent
theory $T$ without equality.
We assume a fixed but arbitrary signature $\Sigma$ of a first-order language.
Phrases like \emph{term, atom, formula, sentence} refer to the well known
syntactic entities.
Let $\XX=\set{x_0,x_1,\ldots}$ be a countably infinite set of variables;
metavariables $x,y,z$ range over $\XX$.
Let $\Tm(X)$ denote the set of $\Sigma$-terms over $X \subseteq \XX$.

\begin{definition}\label{def:Ctermsubst}
The category $\Cts$ has objects denoted as pairs $(X;A)$,
where $X$ is a finite subset of $\XX$ and $A$ is a finite set
of atoms in the language defined by $\Sigma$ and $X$.
The latter means that only variables from $X$ may occur in $A$.
Such pairs $(X;A)$ are called \emph{conditions}.
The morphisms of $\Cts$ are denoted as $f : (Y;B) \to (X;A)$,
where $f$ is a \emph{term substitution} $X \to \Tm(Y)$ such that $Af \subseteq B$.
Here and below, $Af$ denotes the application of the
substitution $f$ to $A$.
\end{definition}

Clearly, $\Cts$ is a category. We have  $\id_{(X;A)} :  (X;A) \to (X;A)$
with underlying function the inclusion $X \to \Tm(X)$.
If $f: (Z;C) \to (Y;B)$ and $g : (Y;B) \to (X;A)$,
then $g\circ f : (Z;C) \to (X;A)$ is defined by the
composition of substitutions $gf : X \to \Tm(Z)$ which sends a
variable $x\in X$ to the term $xgf \in \Tm(Z)$.
Here and below we use diagram composition for the substitutions
underlying the morphisms.
This works well in combination with postfixing substitutions:
$A(g \circ f) = A(gf) = (Ag)f$, where
one can read $gf$ both as the ordinary composition of morphisms and 
as the diagram composition of the underlying substitutions.

Conditions will be denoted as, e.g., $(x,y,z;p(z),q(f(x),z,z))$.
Substitutions will be denoted as, e.g., $[y:=x,z:=g(x)]$.
The category $\Cts$ has a terminal object $(;)$
with unique morphisms $[] : (X;A) \to (;)$.

\begin{definition}\label{def:Cvarsubst}
The category $\Cvs$ is the subcategory of $\Cts$ that
has the same objects but in which the morphisms $f : (Y;B) \to
(X;A)$ are required to be \emph{variable substitutions}.
These are simply functions $X \to Y$ thought of as functions $X \to \Tm(Y)$. 
\end{definition}

Note that if the signature is purely relational,
then the categories $\Cts$ and $\Cvs$ coincide.

\begin{definition}\label{def:Crenaming}
The category $\Crn$ is the subcategory of $\Cvs$ that
has the same objects but in which the morphisms are
required to be injective
(when considered as functions between sets of variables).
Such variable substitutions
are commonly called \emph{renamings}, whence the
subscripts in $\Cvs,\Crn$.
\end{definition}

\begin{discussion}
The above definitions clearly take the nominal approach to conditions.
In the nominal approach it is important to avoid name conflicts.
Also, the nominal approach gives rise to many isomorphisms,
such as between $(x_0;p(x_0))$ and $(x_1;p(x_1))$. Such isomorphisms
do no harm, but are not very useful either.
Alternatively, it would be possible to take $\XX=\nat$ and use
de~Bruijn indices~\cite{deBruijn1972}.
The two conditions above would then become one condition $(1;p(0))$.
In the approach with de Bruijn indices the conditions are
of the form $(n;A)$ with all numbers occurring in $A$ less than $n$.
Morphisms $(m;B) \to (n;A)$ are functions $f: [n] \to \Tm([m])$
such that $Af \subseteq B$.
Here $[k]$ denotes the set $\set{0,\ldots,k-1}$, for each $k\in\nat$.
With de Bruijn indices there are no name conflicts.
However, de Bruijn-indices are more difficult to read and we shall
only use them to resolve name conflicts.
\end{discussion}

\begin{discussion}
  In the above definitions we assumed the language to be
  single-sorted, and we shall write most of the paper under this
  assumption, though very little would have to change to accommodate
  multiple sorts. In this case, a condition $(X;A)$ would consist of a
  finite set $X$ of sorted variables, and a finite set of atoms with
  variables from $X$, and morphisms would be sort-respecting
  (injective, variable, term) substitutions.
\end{discussion}

\subsection{Finite limits}\label{sec:finite_limits}

Both $\Cvs$ and $\Crn$ inherit the terminal object $(;)$ from $\Cts$. In this section
we show that $\Cvs$ actually has all finite limits. We also note that
$\Cts$ has all finite products.
However, neither $\Cts$ nor $\Crn$ have equalizers: For $\Cts$,
consider the two arrows $[x:=0],[x:=1]:(;)\rightrightarrows(x;)$ over a signature
with constants $0,1$. They cannot by equalized.
For $\Crn$, consider the two arrows
$[x:=y],[x:=z] : (y,z;)\rightrightarrows(x;)$; they are equalized by the
non-injective variable substitution $[y:=w,z:=w]:(w;)\to(y,z;)$.

To see that $\Cvs$ has all finite limits, it suffices to check that in
addition to a terminal object it has pullbacks. So let
$(Y;B) \xrightarrow{f}{} (X;A) \xleftarrow{g}{} (Z;C)$ be a
cospan in $\Cvs$. Let $W := Y \sqcup^X Z$ be a pushout of the span
$Y \leftarrow X \to Z$ with corresponding map $i:Y\to W$ and
$j:Z\to W$. Note that $i$ and $j$ are injective if $f:X\to Y$ and
$g:X\to Z$ are. Let $D := Bi\sor Cj$. Then the following is a pullback
square of conditions:
\begin{equation}\label{eq:pbsquare}
  \begin{tikzcd}
    (W;D) \arrow{r}{j}\arrow{d}[swap]{i} &
    (Z;C)\arrow{d}{g} \\
    (Y;B) \arrow{r}[swap]{f} & (X;A)
  \end{tikzcd}
\end{equation}
Indeed, given any other span of conditions
$(Y;B) \xleftarrow{i'}{} (W';D') \xrightarrow{j'}{} (Z;C)$ such that
the analogous square commutes, there is a unique function $h:W\to W'$
factoring the corresponding functions, and this gives a morphism of
conditions $h : (W';D')\to(W;D)$.

To see that $\Cts$ has all finite products, consider the
square~\eqref{eq:pbsquare} in the case where $(X;A)$ is the terminal
object $(;)$ and $W$ is then the disjoint union of $X$ and $Y$.

\subsection{Coverages}\label{sec:coverages}

Recall that a coherent theory $T$ has an axiomatization in which every
axiom is of the form
\[\forall \vec x.\,(\phi_0 \to  \exists\vec x_1.\phi_1 \lor\cdots\lor
  \exists\vec x_k.\phi_k),\]
where the $\phi_i$ are conjunctions of atoms.
\begin{definition}\label{def:cov_T}
Let $T$ be a coherent theory. We define an inductively generated relation $\covT$
as in Definition~\ref{def:Gbasis} for the category $\Crn$.
Since $\Crn$ is a subcategory of $\Cts$ and $\Cvs$ containing all isomorphisms,
we have then also defined relations $\covT$ for $\Cts$ and $\Cvs$.
In Theorem~\ref{thm:cov_T} we will verify that $\covT$ is a saturated coverage
for all three categories, and even a basis for a Grothendieck topology on $\Cvs$.
\[
\frac{\text{$f$ is an isomorphism with codomain $(X;A)$}}
{(X;A)\covT\set{f}}
\]
\[
\frac{(X,\vec x_1; A,\phi_1) \covT U_1 \quad \ldots\quad (X,\vec x_n; A,\phi_n) \covT U_n}
{(X;A) \covT \bigcup_{1\leq i\leq n} (e_i U_i)}(*)
\]
Here $(*)$ is the set of conditions sanctioning the application of the
rule: the existence of an instance
$\phi_0 \to \exists\vec x_1. \phi_1 \lor\cdots\lor \exists\vec x_n. \phi_n$ 
of an axiom in $T$ with all free variables in $X$ and
$\phi_0 \subseteq A$.  We tacitly assume that name conflicts are
avoided, either by using de Bruijn indices, or by renaming the bound
variables $\vec x_i$ so that they are disjoint from $X$.  The
morphisms $e_i : (X,\vec x_i; A,\phi_i) \to (X;A)$ restrict to
identities on $X$, and $e_iU_i$ denotes the collection of morphisms
$e_if$ with $f\in U_i$. This completes $(*)$ and the definition of  $\covT$.
\end{definition}

\begin{lemma}\label{lem:iso_covT}
Let $T$ be a coherent theory and $\covT$ the relation as in Definition~\ref{def:cov_T}.
If $(X;A) \covT U$, then $(Y;B) \covT fU$ for every isomorphism $f: (X;A) \to (Y;B)$.
\end{lemma}

\begin{proof}
Since $T$ is fixed, we drop the subscript in $\covT$.
By induction on the definition of $(X;A) \covd U$.
The base case is trivial, since composition preserves isomorphisms.
For the induction step, consider
\[
\frac{(X,\vec x_1; A,\phi_1) \covd U_1 \quad \ldots\quad (X,\vec x_n; A,\phi_n) \covd U_n}
{(X;A) \covd \bigcup_{1\leq i\leq n} (e_i U_i)}(*)
\]
and assume the lemma holds for the premises of the rule.
Let $f: (X;A) \to (Y;B)$ be an iso. We prove the desired result by the following inference:
%induction step in the definition of $\covd$:
\[
\frac{(Y,\vec x_1; B,\phi_1 f_1) \covd f_1 U_1 \quad \ldots
\quad (Y,\vec x_n; B,\phi_n f_n) \covd f_n U_n}
{(Y;B) \covd \bigcup_{1\leq i\leq n} (e'_i f_i U_i)}(*')
\]
Here every $f_i: (X,\vec x_1;A,\phi_1) \to (Y,\vec x_1;B,\phi_1 f)$ is an isomorphism extending $f$
with identity substitutions for the $\vec x_i$. The $e'_i$ do for $Y$ what the 
$e_i$ do for $X$. The justification $(*')$ of the inference is based on
$(\phi_0 \to \exists\vec x_1. \phi_1 \lor\cdots\lor \exists\vec x_n. \phi_n)f$.
For example, $\phi_0 \subseteq A$ implies $\phi_0 f \subseteq Af \subseteq B$. 
We may waive name conflicts since we could have used de Bruijn indices.
It remains to prove $f e_i U_i = e'_i f_i U_i$ for all $i$. 
This follows since, basically, the $e_i$ do nothing on  $X$, 
the $e'_i$ do nothing on $Y$, and the $f_i$ extend $f$
(actually, we have $f e_i = e'_i f_i$).
\end{proof}

\begin{theorem}\label{thm:cov_T}
Let $T$ be a coherent theory and $\covT$ the relation as in Definition~\ref{def:cov_T}.
The relation $\covT$ is a saturated coverage for any of the categories $\Crn,\Cvs$,$\Cts$
and a basis for a Grothendieck topology on $\Cvs$.
\end{theorem}
\begin{proof} Since $T$ is fixed, we drop the subscript in $\covT$.  We have
to verify conditions \ref{def:Gbasis-i}, \ref{def:Gbasis-ii'} and
\ref{def:Gbasis-iii} in Definition~\ref{def:Gbasis}.

Condition \ref{def:Gbasis-i} holds as per definition of the base case of $\covT$.

We do \ref{def:Gbasis-iii} before \ref{def:Gbasis-ii'} since
\ref{def:Gbasis-iii} is easier. We prove \ref{def:Gbasis-iii} by
induction on the definition of $C \covd U$. The base case, where
$U = \set{f}$ for some isomorphism $f$, is Lemma~\ref{lem:iso_covT},
which holds for all three $\Crn,\Cvs,\Cts$.
For the induction step, consider
\[
\frac{(X,\vec x_1; A,\phi_1) \covd U_1 \quad \ldots\quad (X,\vec x_n; A,\phi_n) \covd U_n}
{(X;A) \covd \bigcup_{1\leq i\leq n} (e_i U_i)}(*)
\]
and assume \ref{def:Gbasis-iii} holds for the premises of the rule.
Assume $C_{ij}\covd V_{ij}$ for all $g_{ij} : C_{ij} \to (X;A)$ in $e_i U_i$, $1\leq i\leq n$.
Let $f_{ij} \in U_i$, then $g_{ij}= e_i f_{ij} \in e_i U_i$, 
and we can use $C_{ij}\covd V_{ij}$ also for $f_{ij}$.
Hence by the induction hypothesis
$(X,\vec x_i; A,\phi_i) \covd  \bigcup_{f_{ij}\in U_i} {f_{ij} V_{ij}}$, $1\leq i\leq n$.
With these covers in hand we prove the desired result by the following inference:
\[
\frac{(X,\vec x_1; A,\phi_1) \covd  \bigcup_{f_{1j}\in U_1} f_{1j} V_{1j}
 \quad \ldots\quad (X,\vec x_n; A,\phi_n) \covd  \bigcup_{f_{nj}\in U_n} f_{in} V_{in}}
{(X;A) \covd \bigcup_{1\leq i\leq n} (e_i  \bigcup_{f_{ij}\in U_i} {f_{ij} V_{ij}} )}(*)
\]
A routine verification shows that indeed 
\[
\bigcup_{1\leq i\leq n} (e_i  \bigcup_{f_{ij}\in U_i} {f_{ij} V_{ij}} )
= \bigcup_{1\leq i\leq n} \bigcup_{g_{ij}\in e_i U_i} g_{ij} V_{ij},
\]
which completes the proof of \ref{def:Gbasis-iii}.

We continue with \ref{def:Gbasis-ii'}, which we again prove by induction 
on the definition of $C \covd U$. In the base case we have that
$C \covd \set{f}$ for some isomorphism $f : C' \to C$.
For any $g: D \to C$ we take $V =  \set{\id_D} \covs D$.
Then it suffices to show that $g = g \id_D$ factors through $f$. 
This is immediate since $g = (g\,k) f = g(kf)$ for the inverse $k$ of $f$.
For the induction step, consider:
\[
\frac{(X,\vec x_1; A,\phi_1) \covd U_1 \quad \ldots\quad (X,\vec x_n; A,\phi_n) \covd U_n}
{(X;A) \covd \bigcup_{1\leq i\leq n} (e_i U_i)}(*)
\]
and assume \ref{def:Gbasis-ii'} holds for the premises of the rule.
Let $g: (Y;B) \to (X;A)$ and define $g_i: (Y,\vec x_i; B,\phi_i g) \to (X,\vec x_i; A,\phi_i)$
extending $g$ with $\vec x_i g = \vec x_i$ , for each $1\leq i\leq n$.
By the induction hypothesis there exist $V_i \covs (Y,\vec x_i; B,\phi_i g)$ for $1\leq i\leq n$
with exactly the properties we need. Now consider the following inference:
\[
\frac{(Y,\vec x_1; B,\phi_1 g) \covd V_1 \quad \ldots\quad (Y,\vec x_n; B,\phi_n g) \covd V_n}
{(Y;B) \covd \bigcup_{1\leq i\leq n} (e_i' V_i)}(*')
\]
One could say that this inference is the pullback along $g$ of the previous one.
Like in the proof of Lemma~\ref{lem:iso_covT} one easily verifies $(*')$. 
%For example, $\phi_0 \subseteq A$ implies $\phi_0 f \subseteq Af \subseteq B$. 
%We may waive name conflicts since we could have used de Bruijn indices.
We define $V = \bigcup_{1\leq i\leq n} (e_i' V_i)$ and the following
commutative diagram (\ref{eq:coverage}) shows that $V$ has
the desired properties.
\begin{equation}\label{eq:coverage} %\circlearrowleft
  \begin{tikzcd}
    C'\arrow{d}{k}\arrow{r}{h_i \in V_i} &
    (Y,\vec x_i; B,\phi_i g)\arrow{d}{g_i} \arrow{r}{e'_i} &
    (Y;B)\arrow{d}{g} \\
    C \arrow{r}[swap]{f_i \in U_i} & (X,\vec x_i; A,\phi_i) \arrow{r}[swap]{e_i} & (X;A)
  \end{tikzcd}
\end{equation}

We have now proved \ref{def:Gbasis-i}, \ref{def:Gbasis-ii'} and \ref{def:Gbasis-iii}
uniformly for $\Crn,\Cvs,\Cts$.
We proceed to the proof of \ref{def:Gbasis-ii} for $\Cvs$, which is
very similar to the proof above of \ref{def:Gbasis-ii'}.
Again we use induction on the definition of $C \covd U$. In the base case we have that
$C \covd \set{f}$ for some isomorfism $f : C' \to C$. We put $V = \set{g^* f}$
and get $D \covd V$ as pullback preserves isos. In the induction step the only
differences with the proof of (ii') are:
the covers $V_i = \set{g_i^* f_i \mid f_i \in U_i}$ obtained by the induction hypothesis,
and the diagram  (\ref{eq:coverage}), which is now a pullback diagram
since the subdiagrams are pullbacks.
\end{proof}

\subsection{Canonical models}\label{sec:canonical-models}

We have now defined, for any fixed coherent theory $T$,
sheaf toposes $\Sh(\C,\covd)$ for $\C=\Cts,\Cvs,\Crn$, 
and we proceed to
define a \emph{canonical model} of $T$ in each of these.
\begin{definition}\label{def:can-model}
  Let $\Tm$ denote the presheaf of terms over the signature in each
  $\Psh(\C_\_)$, so $\Tm(X;A)$ is the set of terms with parameters in $X$.
  Note that for $\Cts$ we have $\Tm \cong \yo(x;)$, but not for $\Cvs,\Crn$.
  Now let $M := \as\Tm$ denote the associated sheaf.

  If $f$ is an $n$-ary function symbol, define $\denote f: M^n\to M$
  by applying the associated sheaf functor to the natural
  transformation $\Tm^n\to\Tm$ that maps $n$ terms $t_i$ at a
  condition $(X;A)$ to the term $f(t_1,\dots,t_n)$.

  If $p$ is an $n$-ary predicate symbol, define $\denote p$
  as the associated sheaf of the sub-presheaf of $\Tm^n$ that at a
  condition $(X;A)$ consists of those tuples $(t_1,\dots,t_n)$ for
  which $p(t_1,\dots,t_n)\in A$.
\end{definition}

\begin{remark}\label{rem:base-cases}
  By the definition of the associated sheaf functor in terms of
  locally compatible families, this means that every element
  $\alpha\in M(X;A)$ is locally a term: there is a cover
  $U\covs(X;A)$ such that for every $f:(Y;B)\to(X;A)$ in $U$,
  $\alpha f=\eta(t)$ for some $t\in\Tm(Y;B)$, and such a family of
  terms defines an element of $M(X;A)$ if it is locally compatible.

  Likewise, the elements $\alpha_1,\dots,\alpha_n\in M(X;A)$ satisfy
  $(X;A)\forces_M p(\alpha_1,\dots,\alpha_n)$ for $p$ an $n$-ary
  predicate symbol, if and only if there is a cover $U\covs(X;A)$ such
  that for every $f:(Y;B)\to(X;A)$ in $U$, we have
  $\alpha_if=\eta(t_i)$ for some terms $t_1,\dots,t_n\in\Tm(Y;B)$ and
  $p(t_1,\dots,t_n)\in B$.
\end{remark}

Note in particular that $(X;A)\forces_M\phi$ if $\phi\subseteq A$ when
$\phi$ is a conjunction of atoms (here and in the sequel we elide the
difference between a term $t$ and the corresponding element $\eta(t)$).
Regarding the canonical interpretation of equality, if
$s,t\in\Tm(X;A)$, then $(X;A)\forces_M s=t$ if{f} there is a cover
$U\covs(X;A)$ such that $sf=tf$ for all $f:(Y;B)\to(X;A)$ in $U$.

We need to check that these structures $M$ actually give models of
$T$. Consider an axiom of $T$:
\[
  \phi = \forall \vec x.\,(\phi_0 \to  \exists\vec x_1.\phi_1 \lor\cdots\lor
  \exists\vec x_k.\phi_k),
\]
which we take to be closed, where the $\phi_i$ are conjunctions of
atoms. We must show $(;) \forces_M \phi$, which by
Theorem~\ref{thm:forcing} amounts to showing for any condition $(X;A)$
and any instantiation of $\vec x$ over $X$ such that the atoms in
$\phi_0$ lie in $A$, we have
\[
  (X;A) \forces_M \exists\vec x_1.\phi_1 \lor\cdots\lor \exists\vec x_k.\phi_k .
\]
To construct a suitable cover of $(X;A)$ we use the same instantiation
in the step case of Definition~\ref{def:cov_T} and obtain a
cover $\{e_i\}_{1\le i\le k}\covs(X;A)$
using identity covers in the premises. Indeed, for each $i$ we have
\[(X,\vec x_i; A,\phi_i) \forces_M \exists\vec x_i.\phi_i,\]
with the identity instantiation of the $\vec x_i$ because 
$(X,\vec x_i;A,\phi_i)\forces_M \phi_i$.

\section{Comparison with the classifying topos}\label{sec:comparisons}

In this section we first recall in Subsection \ref{sec:classifying-toposes} some
background on classifying toposes as well as a convenient presentation
due to Coste and Coste \cite{CosteCoste1975}. We then prove that for a
purely relational signature, $\Sh(\Cvs,\covd_T)$ is the classifying
topos of $T$.
In Subsection~\ref{sec:comparison_morphisms} we study the geometric
morphisms between the sheaf toposes corresponding to our sites.

\subsection{Classifying toposes}
\label{sec:classifying-toposes}

The theory of classifying toposes is based on theories in logic with
equality, so we need to connect theories in logic without equality to
ones in logic with equality. Given a coherent theory $T$ in logic
without equality, let $T^=$ denote $T$ expanded with an equality
relation and the (coherent) axioms ensuring this is a congruence with
respect to the signature.

To discuss classifying toposes, let now $T$ be a coherent theory in
logic with equality (in particular, the signature contains $=$ and all
equality axioms are assumed).
Let $\Mod_T(\mathcal E)$ denote the category of models of $T$ in
$\mathcal E$ and homomorphisms between them. By the remark at the end
of Section~\ref{sec:structures}, $\Mod_T(\mathcal E)$ 
depends on $\mathcal E$ in a contravariant way for a coherent theory $T$, 
as the inverse image part of a geometric morphism
carries $T$-models to $T$-models.

\begin{definition}\label{def:classifying-topos}
  A \emph{classifying topos} for $T$ is a sheaf topos denoted
  $\Set[T]$ equipped with a \emph{generic model} $M_T$ of $T$, i.e.,
  the map $F \mapsto F^*M_T$ induces an equivalence of categories
  \[
    \Hom(\mathcal E, \Set[T]) \simeq \Mod_T(\mathcal E)
  \]
  for every sheaf topos $\mathcal E$.
\end{definition}
Given $T$, the classifying topos and the generic model are
by this universal property
uniquely determined up to unique equivalence and isomorphism.%
  \footnote{It may of course happen for different
  theories $T,T'$ that nevertheless $\Set[T]\simeq\Set[T']$, in which
  case we say that the theories are \emph{Morita-equivalent}.}

When $T$ is a coherent theory in logic without equality, we say that a
topos is a classifying topos for $T$ if it is a
classifying topos for $T^=$. This makes sense because any model of $T$
in a topos $\mathcal E$ is uniquely also a model of $T^=$ in the sense
of logic with equality, as equality must be interpreted by the
equality of $\mathcal E$.

There are several possible constructions of the classifying topos of a
coherent theory $T$ (in logic with equality).  We shall use a
presentation that goes back to \cite[p.~25]{CosteCoste1975}, where the
classifying topos is constructed relative to a subtheory $T^0$ that is
required to be \emph{strict universal Horn}, i.e., consisting only of axioms
of the form (the universal closure of) $\phi_1\land\cdots\land\phi_n \to \psi$, where the
$\phi_i$ and $\psi$ are atomic formulas.%
\footnote{A general \emph{Horn} theory also allows axioms where $\psi$
  is $\bot$.}  The site is based on a category called the
\emph{syntactic site} of $T^0$, $\C^0$. This has as objects pairs
$(X;A)$ where $X$ is a finite set of variables, and $A$ is a set of
facts which may now include equalities.  The morphisms $(X;A)\to(Y;B)$
are equivalence classes of term substitutions $f : Y \to \Tm(X)$ such
that $T^0$ together with the facts in $A$ proves the facts in
$Bf$. Two term substitutions $f,g$ are equivalent if $T^0$ together
with $A$ proves $yf=yg$ for each $y\in Y$.%
\footnote{The notion of syntactic site applies more generally to
  coherent theories, furnishing another construction of the
  classifying topos with respect to a suitable topology, but in this
  case one takes as morphisms the provably functional relations. For
  strict universal Horn theories the two notions coincide, as can be seen,
  for instance, using the intuitionistic multi-succedent calculus of
  \cite{NegriPlato1998}, where Horn axioms
  $\phi_1\land\cdots\land\phi_n \to \psi$ and the equality axioms are
  included as rules.}
Note that $\C^0$ is a category with all finite
limits (by an argument similar to that of \cite[D1.4.2]{Elephant}).

The theorem of \cite[p.~25]{CosteCoste1975} (see also Proposition~D3.1.10 of
\cite{Elephant}) can now be phrased as follows:
\begin{theorem}\label{thm:CC}
  The classifying topos $T$ can be presented as
  $\Sh(\C^0,\covd_T)$, where $\covd_T$ is generated as in
  Definition~\ref{def:cov_T} using the axioms of $T$ not in $T^0$.
  The generic model is $M_T = \as\yo(x;)$.
\end{theorem}

\begin{theorem}\label{cor:topos_equiv}
  For a purely relational signature, $\Sh(\Cvs,\covd_T)$ is
  canonically equivalent to the classifying topos of $T^=$.
\end{theorem}
\begin{proof}
Let $\C^0$ be the syntactic site from \cite{CosteCoste1975}
over the base theory of equality (eq) in the signature of $T^=$.
One difference between $\Cvs$ and $\C^0$ is that the objects
of the latter may contain equality atoms. By the assumption
on the signature, equality atoms are all of the form $x=y$.
These can easily be eliminated via isomorphisms
$(X; A, x=x) \to (X;A)$ and
$(X,y; A, x=y) \to (X;A[y:=x])$.
Adopting a standardized procedure for eliminating equality atoms
thus yields that the canonical fully faithful functor from $\Cvs$
to $\C^0$ is an equivalence of categories.
By induction on the definition of $C \covd_T U$ one proves that
the equivalence preserves covers both ways.
Hence, the sites $\Cvs$ and $\C^0$ give equivalent sheaf toposes,
$\Sh(\Cvs,\covd_T) \simeq \Sh(\C^0,\covd_T)$.
Now apply Theorem~\ref{thm:CC}.
\end{proof}

In a forthcoming paper we shall characterize the sheaf toposes
$\Sh(\C_x,\covd_T)$ as the classifying toposes of certain extensions
of the coherent theory $T$. As a preview, we note that when the
signature contains function symbols, $\Sh(\Cvs,\covd_T)$ and the
classifying topos of $T^=$ can be quite different, cf.\
Example~\ref{exa:bi-pointed} below.

\subsection{Comparison morphisms}
\label{sec:comparison_morphisms}

Let us note that the toposes we have considered are related by a
series of geometric morphisms. Consider the string of canonical faithful functors
\[
  \begin{tikzcd}
    \Crn \arrow{r}{i} &
    \Cvs \arrow{r}{j} &
    \Cts \arrow{r}{k} &
    \C^0
  \end{tikzcd}
\]
where $\C^0$ is the syntactic site of the theory of equality (eq) in the
signature of $T^=$ (using either of the presentations given in
Section~\ref{sec:classifying-toposes}). Giving each of these
categories the coverage corresponding to a coherent theory $T$, all
the functors have the covering lifting property, so we get a
corresponding string of geometric morphisms
\begin{equation}\label{eq:comparison_morphisms}
  \begin{tikzcd}
    \Sh(\Crn) \arrow{r} &
    \Sh(\Cvs) \arrow{r} &
    \Sh(\Cts) \arrow{r} &
    \Sh(\C^0) = \Set[T]
  \end{tikzcd}
\end{equation}
ending with the classifying topos for $T$ (using Theorem~\ref{thm:CC}).

\begin{theorem}
  The inverse image functors map canonical models to canonical models.
\end{theorem}
\begin{proof}
  Recall that the canonical models are given by $\as\Tm$ for $\Crn$
  and $\Cvs$ and $\as\yo(x;)$ in $\Cts$ and $\C^0$. Note by
  Lemma~\ref{lem:comorphism_as}, we can compute the inverse image of
  these associated sheaves by taking the inverse image of the presheaf
  and then sheafifying.

  Now, for $C\in\Cts$ we have
  \[
    k^*(\yo(x;))(C) = \yo(x;)(kC) = \Hom_{\C^0}(kC,(x;)) \simeq
    \Hom_{\Cts}(C;(x;)) = \yo(x;)(C),
  \]
  using that $k$ is full as well as faithful.
  Thus, $k^*(\as\yo(x;)) \simeq \as\yo(x;)$ in $\Sh(\Cts)$.

  Then, for $C\in\Cvs$ we have $j^*(\yo(x;))(C) = \Hom_{\Cts}(C,(x;)) \simeq
  \Tm(C)$, so $j^*(\as\yo(x;)) \simeq \as\Tm$ in $\Sh(\Cvs)$.

  Finally, for $C\in\Crn$ we have $i^*\Tm(C) \simeq \Tm(C)$, so
  $i^*(\as\Tm) \simeq \as\Tm$ in $\Sh(\Crn)$.
\end{proof}
We see that the compositions through to $\Sh(\C^0)$
in~\eqref{eq:comparison_morphisms} witness that the canonical models
are models of $T$.

Since $i$, $j$ and $k$ preserve covers, we can consider whether they
are also covering-flat and thus morphisms of sites, cf.~Theorem~\ref{thm:site-morphism}.

\begin{theorem}\label{thm:flatness}
  The inclusion functor $j : \Cvs \to \Cts$ is representably flat, and
  thus $(j_! \dashv j^*)$ gives rise to a geometric morphism from $\Sh(\Cts)$
  to $\Sh(\Cvs)$.

  The composite $kj : \Cvs \to \C^0$ preserves finite limits, and
  there is thus an induced geometric morphism from $\Sh(\C^0)$ to
  $\Sh(\Cvs)$.
\end{theorem}
\begin{proof}
  We first need to show that for any finite diagram $d$ in
  $\Cvs$, any cone over $d$ in $\Cts$ factors through some cone over
  $d$ in $\Cvs$. But $\Cvs$ has finite limits, and we see from
  their explicit description in Section~\ref{sec:finite_limits} that any
  cone over $d$ in $\Cts$ factors through the limiting cone of $d$ in
  $\Cvs$.

  The second statement is clear.
\end{proof}

The considerations in the beginning of Section~\ref{sec:finite_limits}
that show that $\Crn$ and $\Cts$ fail to have equalizers also show
that $i$ and $k$ need not be covering-flat (take an empty theory for
instance).

One might hope that the morphisms in~\eqref{eq:comparison_morphisms}
are \emph{open}, meaning that the inverse image parts preserve first-order logic. For empty
theories the coverages are trivial, and the toposes are presheaf
toposes. However, we have from \cite[Proposition~2.6]{Johnstone1980}
that $f:\C\to\D$ induces an open geometric morphism $\Psh(\C)\to\Psh(\D)$
if{f} for every $C\in\C$ and $\alpha: D\to f(C)$ in $\D$, there is
$\beta:C'\to C$ in $\C$ as well as $\gamma:f(C')\to D$ and $\delta:
D\to f(C')$ in $\D$ with $\gamma\delta=1_D$ and
$\alpha\gamma=f(\beta)$. We construct counterexamples as follows: for
$i:\Crn\to\Cvs$ let $\alpha$ identify two variables; for
$j:\Cvs\to\Cts$ let $\alpha$ be any non-variable term substitution;
for $k:\Cts\to\C^0$ let $D$ contain an equation. Thus, in general none
of the morphisms in~\eqref{eq:comparison_morphisms} are open.

\section{Soundness and Completeness}\label{sec:coherent_completeness}

In this section we prove for each of the canonical models in
$\Sh(\Crn)$, $\Sh(\Cvs)$ and $\Sh(\Cts)$: \emph{soundness} for a class
of formulas that includes both first order intuitionistic formulas and
(infinitary) geometric implications, and \emph{completeness} for
generalized geometric implications.

We write for $\forces\subrn$, $\forces\subvs$ and $\forces\subts$ for
the forcing relations $\forces_M$ for the canonical models $M$ in
$\Crn$, $\Cvs$ and $\Cts$, respectively, or even just $\forces$ when
the category can be inferred from the context.

If $C=(X;A)$ is a condition, we write $\Fact(C)=A$ for the set of
facts in $C$. In Definition~\ref{def:can-model} we introduced the
notation $\Tm(C)$ for the set of terms with free variables in $X$.

We can reformulate Theorem~\ref{thm:forcing} as a \emph{definition},
referring only to actual terms (rather than locally compatible
families of terms), thus giving a forcing relation between conditions
and formulas which can be understood without any
sheaf or topos machinery.

\begin{definition}\label{def:forcing}
  Let $\C$ be one of the categories $\Crn,\Cvs$ or $\Cts$ and let
  ${\covd}$ be any coverage on $\C$.
  For any condition $C=(X;A)$ and any formula
  $\phi$ with free variables in $X$ we define the forcing relation
  $C\forces\phi$ by induction on $\phi$ as follows:

  \begin{enumerate}
  \item $C\forces\top$ always;

  \item $C\forces\bot$ if $C\covd\es$;
    % (alternatively one can say that there is a dynamic proof of $\bot$ from $C$);

  \item $C\forces\phi$ if $\phi$ is a fact and there is $U \covs C$ with
    $\phi f \in \Fact(D)$ for all $f: D\to C$ in $U$;
    % (alternatively one can say that there is a dynamic proof of $\phi$ from $C$);

  \item $C\forces\bigwedge_{i\in I}\phi_i$ if $C\forces\phi_i$ for all $i\in I$;

  \item $C\forces\bigvee_{i\in I}\phi_i$ if there is $U \covs C$ such that
    for all $f: D\to C$ in $U$, $D\forces\phi_i f$ for some $i\in I$;

  \item $C\forces\phi_1 \to \phi_2$ if for all maps $f : D\to C$
    in $\C$ we have $D\forces\phi_2 f$ whenever $D\forces\phi_1 f$;

  \item $C\forces\forall x.\phi$ if for all maps $f : D \to C$ we have 
    $D\forces\phi [f, x:=t]$ for all $t\in\Tm(D)$;
 
  \item $C\forces\exists x.\phi$ if there is $U\covs C$ such that,
    for all $f: D\to C$ in $U$, $D\forces\phi [f, x:=t]$ for some $t\in\Tm(D)$.

  \end{enumerate}
\end{definition}

\begin{theorem}\label{thm:forcing_agree}
  If we take ${\covd}$ to be ${\covd_T}$ as in
  Definition~\ref{def:cov_T}, then this definition of $\forces$ agrees
  with the one for the canonical model as in
  Definition~\ref{def:forcing_structure} for all formulas.
\end{theorem}
\begin{proof}
  This follows from local character in the base case, and
  Theorem~\ref{thm:forcing} in the other cases, together with
  Lemma~\ref{lem:forcing-quant} in the quantifier cases.
\end{proof}
\begin{remark}\label{rem:forcing_equality}
  If we add to Definition~\ref{def:forcing} the clause
  \begin{enumerate}
    \addtocounter{enumi}{8}
    \item $C\forces s=t$ if there is a cover $U \covs C$ with $sf =
      tf$ for all $f : D \to C$ in $U$,
  \end{enumerate}
  then Theorem~\ref{thm:forcing_agree} remains true when the
  interpretation of equality for $M$ is the equality of the corresponding
  topos. We can thus reason about equality in the canonical models
  even though $T$ is a theory in logic without equality. However, our
  completeness theorem below does \emph{not} hold for sentences
  formulated in logic with equality. Below we shall describe what happens for
  sentences with $=$ via a suitable extension of $T$.
\end{remark}
The only difference between the forcing relations of
Definitions~\ref{def:forcing_structure} and~\ref{def:forcing} is thus
that the former is defined for a wider class of formulas, namely those
containing locally compatible families of terms in place of ordinary
terms. But by local character, we can always reduce to the case of
ordinary terms and formulas.

Having established Theorem~\ref{thm:forcing_agree}, we inherit
the properties of monotonicity and local character also for the
forcing relation of Definition~\ref{def:forcing}.

%For a set of formulas $\Gamma$, let $C\forces\Gamma$ denote 
%$C\forces\phi$ for all $\phi\in\Gamma$.

% Repeated applications of the definition of forcing can sometimes be simplified,
% as stated in the following lemma.
% 
% \begin{lemma}\label{lem:forcing_iter}
%   For all conditions $C$ and formulas $\phi,\psi$ we have:
%   \begin{enumerate}
%   \item $C\forces\forall x y. \phi$ if{f} for all $g : D \to C$ and for
%     all $t_x,t_y \in \Tm(D)$, $D\forces\phi [g, x:= t_x, y:= t_y]$;
%   \item $C\forces\exists x y. \phi$ if{f} there exists $U \covs C$
%     such that for all $f : D \to C$ in $U$ there exist
%     $t_x,t_y \in \Tm(D)$ with $D\forces\phi [f, x:= t_x, y:= t_y]$;
%   \item $C\forces\forall x y.\, \phi\to\psi$ if{f} for all
%     $g : D \to C$ and for all $t_x\in \Tm(D)$,
%     $D\forces\psi [g, x:= t_x]$ whenever $D\forces\phi [g, x:= t_x]$.
%  \end{enumerate}
% \end{lemma}
% \begin{proof}
%   All proofs from left to right are trivial. Conversely, by iterating
%   the definition of forcing, $C\forces\forall x y. \phi$ if{f} for all
%   $g : D \to C$ and for all $t_x \in \Tm(D)$, for all $h : E \to D$
%   and for all $t_y \in \Tm(E)$ we have
%   $E\forces\phi [g, x:= t_x][h,y:= t_y]$. The latter composition of
%   substitutions is $[gh, x:= t_x h,y:= t_y]$, with $gh : E \to C$ and
%   $t_x h, t_y \in \Tm(E)$.  Now the first point follows
%   immediately. We omit the remaining proofs, which are very similar.
% \end{proof}

Sometimes, the case of the universal quantifier simplifies as follows.

\begin{lemma}\label{lem:forcing_forall}
  For $\Cts$, and hence in case of a purely relational signature also for $\Cvs$,
  we have for all conditions $(X;A)$ and formulas
  $\phi$ (possibly with $=$) that
  $(X;A)\forces\forall x.\phi$ if{f} $(X,x;A)\forces\phi$.
\end{lemma}
\begin{proof}
  This is just the last point of Lemma~\ref{lem:forcing-quant} since
  the assumptions imply that $M=\as\yo(x;)$ and the product of $(X;A)$
  with $(x;)$ is $(X,x;A)$ (we may assume $x\notin X$). We can also
  give a direct proof as follows:
  
  The left to right direction is immediate using the inclusion $X
  \hookrightarrow X\sor\{x\}$. For the other direction, assume
  $(X,x;A)\forces\phi$ and let $f : (Y;B)
  \to (X;A)$ in $\Cts$ be given together with a term $t \in \Tm(Y;B)$.
  We then have a commutative triangle
  \[
    \begin{tikzcd}
      (Y;B)\arrow{dr}{f}\arrow{d}[swap]{[f,x:=t]} & \\
      (X,x;A)\arrow{r} & (X;A)
    \end{tikzcd}
  \]
  so $(Y;B)\forces \phi[f,x:=t]$ by monotonicity.
  In case of $\Cvs$ with a purely relational signature, $t$ is a variable in $Y$
  and we can use the same commutative triangle.
\end{proof}
Note that in the case of $\Crn$ and a purely relational signature this
argument would fail in the case of $t$ being a variable in the image
of $f$, because $[f,x:=t]$ would be not be injective.

We mentioned already in Section~\ref{sec:PrelimSheafTopos} that the
forcing semantics is sound with respect to intuitionistic provability.
Let us here note that this is with respect to provability that handles
an empty domain in a correct way.
For example, let $T$ be the empty theory over
a signature containing a constant $c$.  Consider the trivial covering
defined by $(;) \covd \set{\id_{(;)} : {(;)} \to {(;)}}$.  We have
$(;)\forces\top[x=c]$, so by the definition of forcing
$(;)\forces \exists x.\top$.  The latter does not hold if the
signature does not contain a constant.  This means that we have to use
a logic that is careful about $\exists$-introduction, so that
$\exists x.\top$ is not provable if there is no constant in the
signature.  Similar care has to be taken in connection with $\forall$.
Proof systems that can handle empty domains can be found in
\cite{LambekScott1986,Elephant,MostowskiJSL1951,PeremansMC1949}.  One
way is to define triples $\Gamma\ndX\phi$ where $X$ is a set of
variables and $\Gamma,\phi$ are formulas with all free variables in
$X$.  The rules for the propositional connectives are as usual, and
the rules for the quantifiers are as follows.
\begin{enumerate}

\item\label{rule:forall_intro}
$\Gamma\ndX\forall x.\phi$ if $\Gamma\ndXx\phi$ and $x$ not free in $\Gamma$

\item $\Gamma\ndX\exists x.\phi$ if $\Gamma\ndX\phi[x:=t]$ and $t\in \Tm(X)$

\item $\Gamma\ndX\phi[x:=t]$ if $\Gamma\ndX\forall x.\phi$ and $t\in \Tm(X)$

\item $\Gamma\ndX\phi$ if $\Gamma\ndX\exists x.\psi$ and $\Gamma,\psi\ndXx\phi$ 
and $x$ not free in $\Gamma,\phi$
\end{enumerate}
We shall prove soundness for the full class of possibly infinitary formulas,
including (small) infinitary disjunctions and conjunctions, over
the signature including equality.
We add the following rules for introduction and elimination of
infinitary disjunctions and conjunctions:
\begin{enumerate}\addtocounter{enumi}{4}
\item 
$\Gamma\ndX\bigvee_{i\in I}\phi_i$ if $\Gamma\ndX\phi_i$ for some $i\in I$.
\item\label{rule:inf_disj_elim}
$\Gamma\ndX\psi$ if $\Gamma\ndX\bigvee_{i\in I}\phi_i$ and
$\Gamma,\phi_i \ndX \psi$ for all $i\in I$.
\item\label{rule:inf_conj_intro}
$\Gamma\ndX\bigwedge_{i\in I}\phi_i$ if $\Gamma\ndX\phi_i$ for all
$i\in I$.
\item
$\Gamma\ndX\phi_i$ for all $i\in I$ if $\Gamma\ndX\bigwedge_{i\in I}\phi_i$.
\end{enumerate}
Note that by adopting rules \ref{rule:inf_disj_elim} and
\ref{rule:inf_conj_intro} we allow proofs to be infinitary as well.

\begin{theorem}[Soundness]\label{thm:soundness}
  Let $\C$ be one of the categories $\Crn,\Cvs$ or $\Cts$ and let
  ${\covd}$ be a coverage with forcing relation $\forces$.  If
  $\Gamma\ndX\phi$, then for any condition $C$ and any environment
  $\rho : X \to \Tm(C)$ we have $C \forces\phi\rho$ if
  $C \forces\Gamma\rho$, that is, if
  $C \forces\psi\rho$ for all $\psi \in \Gamma$.
\end{theorem}
\begin{proof}
  This of course follows from soundness of intuitionistic logic in sheaf
  semantics (e.g., \cite[D1.3.2]{Elephant} which uses a slightly
  different proof calculus). There is also a direct proof 
  by induction on the derivation of $\Gamma\ndX\phi$
  (for the case of coherent logic and $\Crn$, cf.\ \cite{CoquandLMPS2005}).
  We provide two samples of induction steps: 
  $\forall$-introduction (rule \ref{rule:forall_intro} above), and
  $\bigvee_{i\in I}$-elimination (rule \ref{rule:inf_disj_elim} above).
  In both steps we assume $C$ is a condition and $\rho : X \to \Tm(C)$
  such that $C \forces\Gamma\rho$ (A).
  If $f:D\to C$, then $\rho f : X \to \Tm(D)$ is defined by
  $(\rho f)(x) = \rho(x) f$ so that $(\phi\rho)f = \phi(\rho f)$ for all $\phi$.

Rule \ref{rule:forall_intro}.
  Assume soundness has been proven for $\Gamma \ndXx \phi$, where
  $x$ is not free in $\Gamma$ (IH). We must prove $C\Vdash(\forall x. \phi)\rho$.
  Note that $x$ is not in the domain of $\rho$, so $(\forall x. \phi)\rho$ is $\forall x.(\phi\rho)$.
  By the definition of forcing we have to prove $D\Vdash(\phi\rho)[f,x:=t]$
  for all $f : D\to C$ and $t\in \Tm(D)$. Let $f : D\to C$ and $t\in \Tm(D)$.
  Define $\rho'$ to extend $\rho f$ with value $t$ for variable $x$, so
  $\rho' : X,x \to \Tm(D)$. By (A) and the monotonicity of forcing we get
 $D\Vdash\Gamma\rho f$. Hence $D\Vdash\Gamma\rho'$,
 since $\Gamma\rho'$ is $\Gamma\rho f$ (as $x$ not free in $\Gamma$).
 By (IH) it follows that $D\Vdash\phi\rho'$, which is $D\Vdash(\phi\rho)[f,x:=t]$.

Rule \ref{rule:inf_disj_elim}.
  Assume soundness has been proven for $\Gamma\ndX\bigvee_{i\in I}\phi_i$
  and $\Gamma,\phi_i \ndX \psi$ for all $i\in I$ (IH).
  We must prove $C\Vdash\psi\rho$.
  By (IH) we get $C \Vdash \bigvee_{i\in I}\phi_i \rho$, 
  which by the definition of forcing means that there exists $U \covs C$ 
  such that  for every $f : D\to C$ in $U$ there exists $i\in I$ with $D \Vdash \phi_i\rho f$.
  Let  $f : D\to C$ in $U$, then $D \Vdash\Gamma\rho f$ by 
  (A) and the monotonicity of forcing, 
  and hence $D \Vdash(\Gamma,\phi_i)\rho f$ for some $i\in I$.
  For any such $i$ we get $D \Vdash\psi\rho f$ by (IH). 
  It follows by local character that $C\Vdash\psi\rho$.
\end{proof}

\begin{corollary}\label{cor:soundness}
  Using the coverage $\covd_T$ from Definition~\ref{def:cov_T} we have
  that $(;) \forces \phi$ if $T\ndE \phi$, for all (not necessarily coherent) sentences $\phi$.
\end{corollary}
\begin{proof}
  By Theorem~\ref{thm:soundness} and the fact that $T$ is forced
  relative to $\covd_T$ as in Section~\ref{sec:canonical-models}.
\end{proof}

We now prove the converse of Corollary~\ref{cor:soundness} for
the class of generalized geometric implications,
cf.~Def.~\ref{def:generalized-geometric}.
We first prove the analogue of local character for provability.

\begin{lemma}\label{lem:local_prov}
  Let $\C$ be one of the categories $\Crn,\Cvs$ or $\Cts$ and $T$ a coherent
  theory with corresponding coverage relation $\covd$.  Assume
  $(X;A) \covd U$ and let formula $\psi$ with all free variables in
  $X$ be such that for all $f : (Y;B)\to(X;A)$ in $U$ we have
  $T,B \ndY \psi f$. Then $T,A \ndX \psi$.
\end{lemma}
\begin{proof}
  By induction on $(X;A) \covd U$.  The base case is trivial since the
  substitution underlying an isomorphism is a bijective renaming.  For
  the induction step, let $(X;A) \covd U$ be inferred by
  \[
    \frac{(X,\vec x_1; A,\phi_1) \covd U_1 \quad \ldots\quad
      (X,\vec x_n; A,\phi_n) \covd U_n}
    {(X;A) \covd \bigcup_{1\leq i\leq n} (e_i U_i) = U}(*)
  \]
  using an axiom
  $\phi_0 \to \exists\vec x_1.\phi_1 \lor\dots\lor \exists\vec
  x_k.\phi_k$ with $\phi_0 \subseteq A$.  Assume $T,B \ndY \psi f$ for
  all $f : (Y;B)\to (X;A)$ in $U$, that is, $T,B \ndY \psi e_i h$ for
  all $1\leq i\leq n$ and all $h : (Y;B)\to (X,\vec x_i; A,\phi_i)$ in
  $U_i$.  By induction hypotheses it follows for all
  $1\leq i\leq n$ that $T,A,\phi_i \ndXxi \psi$ ($e_i$ does nothing on
  $X$, so $\psi=\psi e_i$). The axiom used to infer $(X;A)\covd U$
  yields
  $T,A \ndX \exists\vec x_1.\phi_1 \lor\cdots\lor \exists\vec
  x_k.\phi_k$. Note that the $\vec x_i$ do not occur in $T,A,\psi$, so
  from case distinction and existential elimination it follows that
  $T,A \ndX \psi$.
\end{proof}

Now we come to our first completeness result, and here it is
imperative that equality is not present in the language.
\begin{theorem}\label{thm:completeness}
  Let $\C$ be one of the categories $\Crn,\Cvs$ of $\Cts$ and $T$ a
  coherent theory with corresponding forcing relation $\forces$. Then
  for any generalized geometric implication $\phi$ in variables $Y$, any condition $(X;A)$,
  and any environment $\rho:Y\to\Tm(X)$
  we have $T,A \ndX \phi\rho$ if $(X;A) \Vdash\phi\rho$.
\end{theorem}

\begin{proof}
  By induction on $\phi$
  as generated by the grammar in Definition~\ref{def:generalized-geometric}.

   The case of $\top$ is trivial. For the other atomic cases,
    if $(X;A) \forces \alpha$, then
    there exists $U\covs(X;A)$ such that for any $f :
    (Y;B)\to(X;A)$ in $U$ we have $\alpha f \in B$. Obviously, we then
    have $T,B \ndY \alpha f$, so $T,A \ndX \phi$ by
    Lemma~\ref{lem:local_prov}.

   The induction step for $\phi_1 \land \phi_2$ is obvious.

   Assume the result has been proven for $\phi_i$ for $i\in I$ (IH),
   and let $(X;A) \Vdash\bigvee_{i\in I} \phi_i$.
   Then there exists $U \covs (X;A)$ such that for any $f: (Y;B)\to (X;A)$
    in $U$ we have $(Y;B) \forces \phi_i f$ for some $i\in I$.
    By the induction hypothesis we get for every $f$ in
    $U$ that $T,B \ndY\phi_i f$ for some $i\in I$,
    and hence $T,B \ndY \bigl(\bigvee_{i\in I} \phi_i\bigr) f$. 
    Now $T,A \ndX \bigvee_{i\in I}\phi_i$
    follows by Lemma~\ref{lem:local_prov}.

    Assume the result has been proven for $\phi$ (IH),
   and let $(X;A) \Vdash \forall x.\,\phi$. Then
   $(X,x;A) \Vdash \phi$, so by  the induction hypothesis we get
    $T,A \ndXx \phi$, and hence $T,A \ndX \forall x.\,\phi$ since $x$ does not
    occur in $T,A$.

   Again assume the result has been proven for $\phi$ (IH),
   and let $(X;A) \Vdash \exists x.\,\phi$. 
   Then there exists $U \covs (X;A)$ such
    that for any $f: (Y;B)\to (X;A)$ in $U$ we have
    $(Y;B) \forces \phi f [x:=t_x]$ for some
    $t_x \in \Tm(Y)$.  By the induction hypothesis we get
    $T,B \ndY \phi f [x:=t_x]$ and hence
    $T,B \ndY (\exists x.\,\phi) f$. Now $T,A \ndX \phi$ follows by
    Lemma~\ref{lem:local_prov}.
  
   Assume the result has been proven for $\alpha$ and $\phi$ (IH),
   and let $(X;A) \Vdash \alpha\to\phi$. We have to prove
   $T,A \ndX \alpha\to\phi$, and it clearly suffices to prove
   that $T,A,\alpha \ndX \phi$. The latter follows easily from
   the induction hypothesis for $\phi$, the definition of forcing,
   and the fact that $\id_X : (X;A,\alpha) \to (X;A)$.
\end{proof}

\begin{corollary}[Completeness]\label{cor:coherent_complete}
  Let conditions be as above. Then $T\ndE \phi$ if $(;) \forces \phi$,
  for all generalized geometric sentences $\phi$.
\end{corollary}

\begin{remark}
  The completeness result depends crucially on the fact that the
  forcing conditions \emph{do not allow equality facts}.  E.g.,
  $\neg(0=1)$ is forced but not provable, for the empty theory with
  two constants, $0$ and $1$; cf.\ Example~\ref{exa:bi-pointed} below.
\end{remark}

To capture the coherent implications with $=$ that hold in our forcing
models, we let $T^+$ denote $T^=$ together with the following (coherent)
\emph{constructor axioms} ensuring all function symbols behave like
constructors:
\begin{enumerate}[label=(\Roman*)]
\item\label{ax:ts1} distinct function symbols $f,g$ have disjoint
  values: $f(\vec x) = g(\vec y) \to \bot$.
\item\label{ax:ts2} function symbols are injective:
  $f(\vec x) = f(\vec y) \to \vec x = \vec y$.
\item\label{ax:ts3} there are no proper cycles:
  $x = f(\vec s) \to \bot$ whenever $x$ occurs anywhere in the
  sequence of terms $\vec s$.
\end{enumerate}
It is easy to see that these are all forced in our three models, which
thus model $T^+$.

Observe that the base case in the proof of
Theorem~\ref{thm:completeness} also holds with $\alpha$ an equality
atom and $T$ a coherent theory in logic with equality.  (The induction
step $\alpha\to\phi$ fails with $\alpha$ an equality atom since such
$\alpha$'s are not allowed in conditions.) As a consequence,
Theorem~\ref{thm:completeness} also holds for coherent theories $T^+$
and geometric implications with $=$, as long as the latter have no
equations in a negative position.  This allows us to handle geometric
sentences (but not all generalized geometric implications) using the
(admissible) Eriksson-Girard elimination rule for equality, i.e., the
inference
\[
  \infer{\text{\ref{ax:ts1}--\ref{ax:ts3}},\Gamma, \vec s=\vec t \ndX \phi}{
    \text{\ref{ax:ts1}--\ref{ax:ts3}},\Gamma f \ndY \phi f
  }
\]
whenever $f : X \to \Tm(Y)$ is the most general unifier of the equations
$\vec s=\vec t$, cf.~\cite{Eriksson1992,Girard1992}.

\begin{theorem}\label{thm:plus-completeness}
  For all geometric sentences $\phi$ in logic with equality,
  we have that $T^+ \ndE \phi$ if $(;) \forces \phi$, for
  \emph{any} of our forcing models.
\end{theorem}
\begin{proof}
  Any geometric sentence over the signature of
  $T^+$ is equivalent to a conjunction of sentences of the form
  $\forall\vec x.\; \vec s=\vec t \to \psi$, where $\psi$ is a
  generalized geometric implication without any
  equations in an antecedent.

  If the equations $\vec s=\vec t$ are not unifiable, then $T^+ \ndE
  \forall\vec x. \; \vec s=\vec t\to\bot$, and we are done.

  Otherwise, let $f : X \to \Tm(Y)$ be the most general unifier. From
  \[
    (;) \forces \forall\vec x. \; \vec s=\vec t \to \psi,
  \]
  we get $(Y;)\forces (\vec s=\vec t)f \to \psi f$ simply by using the
  term substitution $f$ as instantiation for $\vec x$. Since $f$
  unifies $\vec s=\vec t$, we get $(Y;)\forces \psi f$. By the
  restriction on $\psi$, we get $T^+ \ndY \psi f$ from
  Theorem~\ref{thm:completeness}. By the Girard-Eriksson rule
  we conclude $T^+,\vec s=\vec t \ndX \psi$, as desired.
\end{proof}

\begin{remark}
  Because we assume our theory $T$ to be coherent (viz., finitary
  geometric), we note that in the completeness results,
  Theorem~\ref{thm:completeness} and
  Corollary~\ref{cor:coherent_complete}, the resulting proof is always
  finitary as well, by analyzing the rules used.

  Thus, if in Theorem~\ref{thm:plus-completeness}, $\phi$ has
  the form~\eqref{eq:coherent-implication} (where
  the disjunction is allowed to be infinite), then the proof obtained
  there is finitary.
\end{remark}

We shall see in the companion paper that $\Sh(\Cts,\covd_T)$ is in
fact the classifying topos of $T^+$.
As a hint for what happens for the model over $\Crn$, note that the
generalized geometric implication
$\forall x,y.\, x=y \lor (x=y \to \bot)$ is forced there, so we cannot
hope to extend Theorem~\ref{thm:plus-completeness} to generalized
geometric implications for this model.

\section{Examples}\label{sec:examples}

\begin{example}\label{exa:Schanuel}
  Consider the empty single-sorted signature with the theory
  $\set{\exists x.\top}$ (axiomatizing an inhabited domain). In this case,
  $\Cvs$ is equivalent to $\Fin\op$ (the opposite of the category of finite sets
  and functions) and $\Crn$ is equivalent to $\Finmon\op$ (the
  opposite of the category of finite sets and injections).
  By Theorem~\ref{cor:topos_equiv} we have that
  $\Sh(\Cvs)$ is the classifying topos of the theory $\set{\exists x.\top}^{=}$,
  or, in categorical terminology, for an inhabited object.
  The coverage on $\Crn$ includes all singleton sinks, as any
  injection between finite sets is a composition of finitely many
  embeddings that each add one element with a bijection. Hence the generated
  Grothendieck topology is the atomic one consisting of all inhabited
  sieves, and we get that $\Sh(\Crn)$ is the
  Schanuel topos, classifying infinite decidable objects,
  cf.~\cite[p.~335]{Wraith1980}.
  This example shows that the toposes $\Sh(\Crn)$ and $\Sh(\Cvs)$
  need not be equivalent.
\end{example}

\begin{example}\label{exa:bi-pointed}
  Consider the empty theory over the signature consisting of two
  constant symbols $0$ and $1$. Since the theory is empty, the Grothendieck
  topology of any of the sites considered in this example is trivial,  
  so that Example~\ref{exa:presheaftopos} applies to any of these sites.
  Using \cite[p.~25]{CosteCoste1975},
  the classifying topos can be presented as $\Psh(\C^0)$,
  where $\C^0$ is the syntactic site, see Section~\ref{sec:classifying-toposes}.
  Mapping any object $(X;A)$ of $\C^0$ to the set $X\cup\set{0,1}$
  modulo the equalities implied by $A$,
  it is easy to see that $\C^0$ is equivalent to
  $\Fin_{01}\op$ (the opposite of the category of finite bi-pointed
  sets), and this topos classifies a bi-pointed object. Meanwhile,
  $\Cvs$ and $\Crn$ are impervious to the constants, so
  $\Sh(\Cvs)=\Psh(\Cvs)$ is the object classifier and
  $\Sh(\Crn)=\Psh(\Crn)$ classifies decidable objects~\cite[D3.2.7]{Elephant}.
  Spitters~\cite{Spitters2016} noted that $\Sh(\Cts)$ classifies strictly
  bi-pointed objects (note also that $\Cts$ is equivalent to
  $\Fin_{0\ne1}\op$: the opposite of the category of finite strictly
  bi-pointed sets). This example shows that all of these toposes can
  be different (viz.~non-equivalent). Indeed, consider the string of geometric
  morphisms~\eqref{eq:comparison_morphisms} in this case (${[\C,\D]}$
  denotes the functor category):
  \[
    \begin{tikzcd}
      \Psh(\Crn) \arrow[equal]{d}\arrow{r} &
      \Psh(\Cvs) \arrow[equal]{d}\arrow{r} &
      \Psh(\Cts) \arrow[equal]{d}\arrow{r} &
      \Psh(\C^0) \arrow[equal]{d} \\
      {[\Finmon,\Set]} \arrow{r} &
      {[\Fin,\Set]} \arrow{r} &
      {[\Fin_{0\ne1},\Set]} \arrow{r} &
      {[\Fin_{01},\Set]}
    \end{tikzcd}
  \]
  Here, $\Fin\to\Fin_{0\ne1}$ is the functor that freely adjoins two new
  elements $0$ and $1$. The forgetful functors from $\Fin_{0\ne1}$ and
  $\Fin_{01}$ to $\Fin$ induce the geometric morphisms from
  $\Psh(\Cts)$ and $\Psh(\C^0)$ to $\Psh(\Cvs)$ mentioned in
  Theorem~\ref{thm:flatness}; these classify the generic
  (strictly) bi-pointed object as an object.

  Regarding the canonical models, we have in $\Psh(\C^0)$ that
  $\lnot\lnot 0=1$ using $(X;A,0=1)\to(X;A)$, 
  while in $\Psh(\Cts)$, $\Psh(\Cvs)$ and $\Psh(\Crn)$ 
  we have $\lnot 0=1$.

  In $\Psh(\Cts)$ we have $\forall x.\lnot\lnot(x=0\lor x=1)$,
  while in $\Psh(\Cvs)$ we have on the other hand
  $\lnot\forall x.\lnot\lnot(x=0\lor x=1)$. To see this, use that
  $\Tm(X)$ consists of $0,1$ and $x\in X$, plus that the term
  substitution $[x:=0]$ can be used in $\Cts$ but not in $\Cvs$.

  In $\Psh(\Cvs)$ we have $\lnot\forall xy.\, x=y\lor\lnot x=y$, while
  we have $\forall xy.\, x=y\lor\lnot x=y$ in $\Psh(\Crn)$ 
  using $\Tm(X)$ as above plus the fact that in $\Cvs$ two distinct variables
  can be substituted with one variable, whereas in $\Crn$ they must remain distinct.

  To sum up, these considerations show that the canonical models in
  these toposes are different, at least with respect to sentences
  formulated in logic with equality (and we also get a new proof that
  the morphisms in~\eqref{eq:comparison_morphisms} need not be open).

  To show that the toposes are actually different (i.e.,
  \emph{non-equivalent}), we can use the result that if $\C$ is
  idempotent complete, then it can be recovered up to equivalence from
  the presheaf topos $\Psh(\C)$ \cite[Lemma~A1.1.10]{Elephant}. Now
  note that the four categories $\Finmon$, $\Fin$, $\Fin_{0\ne1}$,
  $\Fin_{01}$ (and their opposites) are all idempotent complete and
  pairwise non-equivalent. Hence the above presheaf toposes are also
  pairwise non-equivalent.
\end{example}

Now we turn to purely relational signatures and show that also there
we find examples of (non-coherent) sentences that are forced but not
provable, not even classically.

\begin{example}\label{exa:forcing_not_all_x_Px}
Let $T$ be the empty theory over a signature with one unary predicate $P$.
Then the formula $\forall x. P(x)$ is never forced.
$(X;A)\forces\forall x. P(x)$ means by the definition of
forcing that $D\forces P(t)$ for all $g : D\to (X;A)$ and $t\in \Tm(D)$.
Since the theory is empty, we only have $D \covd \set{f}$,
for any isomorphism $f : E\to D$, and no other covers.
Now take $D=(X,y;A)$ with $y$ not in $X$,
and $g$ the identity substitution on $X$.
Clearly, there is no isomorphism $f : E\to D$ such that $P(y)f \in Af$,
so not $D\forces P(y)$ and hence not $(X;A)\forces\forall x. P(x)$.
Since $\forall x. P(x)$ is never forced, we get that
$(\forall x. P(x)) \to \bot$ is forced by any condition. 
However, this formula is not classically provable from $T$.
%Further exploiting this example, we get that the formula
%$((\forall x)P(x))\lor\neg((\forall x)P(x))$ is forced by any condition,
%classically true, but not intuitionistically.
\end{example}

The argument of the previous example plays an important
role in the next, more interesting example, which reveals
a difference between forcing based on $\Crn$ and on $\Cvs$.

\begin{lemma}\label{lem:forcing_diff_Crn_Cvs}
  Let $T$ be the theory with one axiom $\forall x,z.\, P(x) \to (Q(x,z) \lor R(x,z))$
  over the minimal relational signature in which this axiom can be
  expressed.  Let
  \[\phi := \exists x,y.~( P(x) \land P(y) \land
    \neg \forall z.\, (Q(x,z) \lor R(y,z))).\]
  Then $\Fvs \neg\phi$, and $\Frn \neg\neg\phi$.
\end{lemma}
\begin{proof}
  We note first that $\bot$ does not occur in $T$.  By a simple
  induction on the definition of $C \covd U$ one can therefore
  construct an $f \in U$ for every $U \covs C$ ($\star$), so we never
  have $C \forces \bot$. By the definition of forcing, $\Fvs \neg\phi$
  means that $C \Fvs \bot$ whenever $C \Fvs \phi$, i.e., that
  $C \Fvs \phi$ is absurd. By the definition of forcing 
  (using transitivity of covers for the  iterated existential quantification), 
  $C \Fvs \phi$ means that for some
  $U \covs C$ we have, for all $f: D \to C$ in $U$,
  \[
    D \Fvs P(d_1) \land P(d_2) \land
    \neg\forall z.\, (Q(d_1,z) \lor R(d_2,z))
  \]
  for some $d_1,d_2 \in \Tm(D)$. By ($\star$) there exist such an
  $f\in U$. It is an intuitionistic consequence of the axiom of $T$
  that $\forall x.\, P(x) \to \forall z.\, Q(x,z)\lor R(x,z)$, so this
  is forced. Hence, taking $x:=d_1$ at $D$ we get
  \[
    D \Fvs \forall z.\, (Q(d_1,z) \lor R(d_1,z)).
  \]
  Since the signature is relational, the terms $d_1,d_2$ are
  variables.  We now use a non-injective substitution
  equating $d_1$ and $d_2$.  Define $g:E \to D$ by $g(d_i)=d_1$ for
  $i=1,2$, and $g$ the identity substitution on all other variables in
  $D$. By monotonicity of forcing it follows that
  \[
    E \Fvs (\forall z.\, Q(d_1,z) \lor R(d_1,z))
    \land \neg\forall z.\, (Q(d_1,z) \lor R(d_1,z)),
  \]
  which implies $E\Fvs\bot$, so by the definition of forcing, $E$ has
  an empty cover.  This conflicts with ($\star$).  Note that we have
  crucially used a substitution that is not injective. This means
  that for $\Frn$ the situation is completely different.

  We prove $\Frn \neg\neg\phi$ by showing that every condition can be
  extended to a condition in which $\phi$ is forced.  Let $C$ be a
  condition and let $D$ denote the extension of $C$ with two new
  variables $x,y$ and corresponding atoms $P(x),P(y)$.  Since
  $D \Frn P(x)\land P(y)$ it suffices to show that
  \[
    D \Frn \neg \forall z.\, (Q(x,z) \lor R(y,z)).
  \]
  Let $g: E\to D$ and assume $E \Frn \forall z.\, (Q(x,z) \lor R(y,z))$.
  This situation is contradictory for a subtle reason.  If $E$ would
  extend $D$ with the atom $x=y$, then there would be no contradiction
  at all. However, the language has no equality. If we would have
  $g(x)=g(y)$, then again there would be no contradiction at all.
  However, $g$ must be injective.  It turns out that a contradiction
  is inevitable under the given conditions.

  Like in Example~\ref{exa:forcing_not_all_x_Px}, extending $E$ with a
  new variable $z$ leads to a condition $F$ that cannot possibly force
  $Q(x,z) \lor R(y,z)$ (for simplicity of notation we assume that all
  renamings are implemented as injections). We must reason carefully
  since the theory $T$ contains one axiom, and we can force both
  $Q(x,z) \lor R(x,z)$ and $Q(y,z) \lor R(y,z)$. The intuition is that
  this gives four possibilities, one of which, $Q(y,z) \land R(x,z)$,
  does not force $Q(x,z) \lor R(y,z)$.  Since
  $F \Frn Q(x,z) \lor R(y,z)$ requires a $U \covs F$ such that for all
  $g: G\to F$ in $U$, either $G \Frn Q(x,z)$ or $G \Frn R(y,z)$, we
  arrive at a contradiction (formally, we argue by induction on $U
  \covs F$).
\end{proof}

Lemma~\ref{lem:forcing_diff_Crn_Cvs} gives a coherent theory $T$ 
and a sentence $\neg\phi$
that is forced in $\Fvs$ and hence, by Theorem~\ref{cor:topos_equiv},
is true in the generic model of $T^=$. Then, of course, $\neg\neg\phi$
is false in this model. At the same time, the lemma states
that $\neg\neg\phi$ is forced in $\Frn$. Hence, by the interplay
between soundness and coherent completeness
(Corollary~\ref{cor:coherent_complete}),
$\neg\neg\phi$ is $T$-redundant, as explained in the Introduction.

At first sight, this theory $T$ and sentence $\neg\neg\phi$ seem to 
answer Wraith's question cited in the introduction. 
However, there is subtle catch here:
Wraith's question has been posed in the context of logic with equality,
and $\Frn \neg\neg\phi$ has been shown in a logic without equality.
Moreover, as pointed out in the proof of Lemma~\ref{lem:forcing_diff_Crn_Cvs},
it is essential for $\Frn \neg\neg\phi$ that equality is not in the language.
%On the other hand, adding equality to the language
%does not affect truth in the generic model of $T^=$.
To answer Wraith's question properly, we have to show that
$\neg\neg\phi$ is $T$-redundant
\emph{in intuitionistic logic with equality}. But this follows from
our improved completeness result, Theorem~\ref{thm:plus-completeness},
because for a relational signature, $T^+$ is the same as $T^=$.

\section{Concluding remarks and open problems}\label{sec:conclusion}

We have presented three simple and natural  
syntactic forcing models for coherent logic,
that do not presuppose that the logic has equality.
In a future paper we shall characterize each of their sheaf toposes
as the classifing topos of a certain extension of the coherent theory
$T$. 
We have already seen a glimpse of this in
Theorem~\ref{thm:plus-completeness}, because $\Sh(\Cts,\covd_T)$ is
the classifying topos of $T^+$.

As an application of our forcing models we have given in 
Section~\ref{sec:examples}
a coherent theory $T$ and a sentence $\phi$ which is
$T$-redundant, yet false in the generic model of $T$
in logic with equality. This answers in the negative a question by
Wraith.

The above example uses predicates other than equality and uses that
logic with equality is conservative over logic without equality.
This leads us to the second open problem, which is the refinement 
of Wraith's question for algebraic coherent theories.
Given an algebraic (= purely equational) coherent theory $T$,
are all $T$-redundant sentences
true in the generic model of $T$?

\section*{Acknowledgements}

We are indebted to Peter Johnstone for information on the generic model.
The first two authors gratefully acknowledge the support of the 
Air Force Office of Scientific Research through MURI grant FA9550-15-1-0053. 
Any opinions, findings and conclusions or recommendations expressed in this material 
are those of the authors and do not necessarily reflect the views of the AFOSR.

\end{document}